\documentclass[10pt, reqno]{amsart}
\usepackage{epsfig,psfrag,amsmath,amssymb,latexsym}
\usepackage{amscd }
\usepackage{amsfonts}
\usepackage{graphicx}
\usepackage[colorlinks=true]{hyperref}
\usepackage{amsthm, amscd,epsfig,enumerate, xcolor}
\usepackage[prependcaption,backgroundcolor=blue!20!white]{todonotes}
\usepackage{caption}
\usepackage{subcaption}
\usepackage{changepage}

\hypersetup{
  colorlinks,
  citecolor=blue,
  linkcolor=blue,
  urlcolor=blue}

\advance\hoffset -.75cm

\newtheorem{theorem}{Theorem}[section]

\newtheorem{lemma}{Lemma}[section]

\newtheorem{proposition}{Proposition}[section]

\numberwithin{equation}{section}

\def\X{\mathcal{ X}}

\def\Z{\mathcal{ Z}}

\def\U{\mathcal{ U}_n}
\def\V{\mathcal{ V}_n}

\def\A{\mathbb{A}}
\def\e{\epsilon}
\def\v{{\rm Var}}
\def\S{\mathcal{ S}_n}

\def\tZ{\mathcal{R}(Z)}
\newcommand{\newparagraph}[1]{\bigskip \\ \textbf{#1}}
\begin{document}
\title{\LARGE Lower bounds for moments of global scores of pairwise Markov chains}

\author[J.~Lember]{J\"uri Lember}
\address{J.~Lember, Institute of Mathematics and Statistics,
University of Tartu, J. Liiv 2, 50409 Tartu, Estonia.}
\email{juri.lember\@@ut.ee}

\author[H.~Matzinger]{Heinrich Matzinger}
\address{H.~Matzinger, School of Mathematics, Georgia Tech, Atlanta (GA), 30332 USA.}

\email{matzi\@@math.getech.edu}

\author[J.~Sova]{Joonas Sova}
\address{J.~Sova,  Institute of Mathematics and Statistics,
University of Tartu, J. Liivi 2, 50409 Tartu, Estonia.}

\email{joonas.sova\@@ut.ee}

\author[F.~Zucca]{Fabio Zucca}
\address{F.~Zucca, Dipartimento di Matematica,
Politecnico di Milano,
Piazza Leonardo da Vinci 32, 20133 Milano, Italy.}
\email{fabio.zucca\@@polimi.it}

\begin{abstract}
Let $X_1,X_2,\ldots$ and $Y_1,Y_2,\ldots$ be two random sequences so
that every random variable takes values in a finite set
$\mathbb{A}$. We consider a global similarity score
$L_n:=L(X_1,\ldots,X_n;Y_1,\ldots,Y_n)$ that measures the homology
(relatedness) of words $(X_1,\ldots,X_n)$ and $(Y_1,\ldots,Y_n)$. A
typical example of such score is the length of the longest common
subsequence. We study the order of central absolute moment
$E|L_n-EL_n|^r$ in the case where the two-dimensional process
$(X_1,Y_1),(X_2,Y_2),\ldots$ is a Markov chain on $\mathbb{A}\times
\mathbb{A}$. This is a very general model involving independent
Markov chains, hidden Markov models, Markov switching models and
many more. Our main result establishes a general condition that
guarantees that $E|L_n-EL_n|^r\asymp n^{r\over 2}$. We also perform
simulations indicating  the validity of the condition.
\end{abstract}

\maketitle

\noindent {\bf Keywords}: Random sequence comparison, longest common sequence, fluctuations,  Waterman conjecture.

\noindent {\bf AMS subject classification}: 60K35, 41A25, 60C05
%

\section{Introduction}
\subsection{Sequence comparison setting}\label{sec:pre}
Throughout this paper  $X=(X_1,X_2,\ldots X_n)$ and
$Y=(Y_1,Y_2,\ldots Y_n)$ are two  random strings, usually referred
as sequences, so that every random variable $X_i$ and $Y_i$ take
values on a finite alphabet $\mathbb{A}$. Since the sequences $X$
and $Y$ are not necessarily independent nor identically distributed,
it is convenient to consider the two-dimensional sequence
$Z=((X_1,Y_1),\ldots,(X_n,Y_n))$. The sample space of $Z$ will be
denoted by $\Z_n$. Clearly $\Z_n\subseteq (\A\times \A)^n$ but,
depending on the model, the inclusion can be strict.
\\
The problem of measuring the similarity of $X$ and $Y$ is central in
many areas of applications including computational molecular biology
\cite{christianini, Durbin, Pevzner, SmithWaterman,
watermanintrocompbio} and computational linguistics
\cite{YangLi,LinOch,Melamed1,Melamed2}. In this paper, we adopt the
same notation as in \cite{LemberMatzingerTorres:2012(2)}, namely we
consider a general scoring scheme, where $S:\mathbb{A}\times
\mathbb{A}\rightarrow \mathbb{R}^+$ is a {\it pairwise scoring
function} that  assigns a score to each couple of letters from
$\mathbb{A}$.
An {\it alignment} is a pair $(\rho,\tau)$ where
$\rho=(\rho_1,\rho_2,\ldots,\rho_k)$ and
$\tau=(\tau_1,\tau_2,\ldots,\tau_k)$ are two increasing sequences of
natural numbers, i.e. $1\leq \rho_1<\rho_2<...<\rho_k\leq n$ and
$1\leq \tau_1<\tau_2<\ldots<\tau_k\leq n.$ The integer $k$ is the
number of aligned letters, $n-k$ is the number of non-aligned
letters. Given the pairwise scoring function $S$ the score of the
alignment $(\rho,\tau)$ when aligning $X$ and $Y$ is defined by
$$U_{(\rho,\tau)}(X,Y):=\sum_{i=1}^kS(X_{\rho_i},Y_{\tau_i})+ \delta (n-k),$$
where $\delta\in \mathbb{R}$ is another scoring parameter. Typically
$\delta\leq 0$ so that many non-aligned letters in the alignment
reduce the score. If $\delta\leq 0$, then its absolute value
$|\delta|$ is often called  the {\it gap penalty}. Given $S$ and
$\delta$, the optimal alignment score of $X$ and $Y$ is defined to
be
\begin{equation}\label{Ln}
L_n:=L(X,Y)=L(Z):=\max_{(\rho,\tau)}U_{(\rho,\tau)}(X,Y),
\end{equation}
where the maximum above is taken over all possible alignments.
Sometimes, when we talk about a {\it string comparison model}, we
refer to the study of $L_n$ for given sequences $X$ and $Y$, score
function $S$ and gap penalty $\delta$.
It is important to note  that for any constant gap price
$\delta\in\mathbb{R}$, changing the value of one of the $2n$ random
variables $X_{1},\ldots,X_{n},Y_{1},\ldots,Y_{n}$ changes the value
of $L_{n}$ by at most $\Delta$, where
\begin{align}\label{eq:MaxScoreChange}
\Delta:=\max_{u,v,w\in {\A}}|S(u,v)-S(u,w)|.
\end{align}
\noindent When $\delta=0$ and the  scoring function assigns one to
every pair of similar letters and zero to all other pairs, i.e.
\begin{equation}\label{LCS-scoring}
S(a,b)=\left\{
           \begin{array}{ll}
             1, & \hbox{if $a=b$;} \\
             0, & \hbox{if $a\ne b$.}
           \end{array}
         \right.
\end{equation}
then $L(Z)$ is just the maximal number of aligned letters, also
called the length of the {\it longest common subsequence}
(abbreviated by LCS) of $X$ and $Y$. In this article,  to
distinguish the length of LCS from another scoring schemes, we shall
denote it via $\ell_n:=\ell(Z)=\ell(X,Y))$. In other words $\ell(Z)$
is the maximal $k$ so that there there exists an alignment
$(\rho,\tau)$ such that $X_{\rho_i}=Y_{\tau_i}$, $i=1,\ldots,k$.
Note that the optimal alignment $(\rho,\tau)$ as well as the longest
common subsequence $X_{\rho_1},\ldots,X_{\rho_k}$ is not typically
unique. The length of LCS is probably the most important and the
most studied measure of global similarity between strings.
\subsection{History and overview}
The problem of measuring the similarity of two strings is of central
importance in many applications including computational molecular
biology, linguistics etc. For instance, in  computational molecular
biology, the similarity of two sequences (for example DNA- or
proteins) is used to determine their  homology (relatedness) --
similar strings are more likely to be the decedents of the same
ancestor. Out of all possible similarity measures, the global score
$L(X,Y)$, in particular the length of LCS, is probably the most
common measure of similarity. Its popularity is  partially due to
the well-known dynamic programming algorithms (so-called
Needleman-Wunsch algorithm) that allows to calculate the optimal
alignment with complexity $O(n^2)$ and the score with complexity
$O(n)$ \cite{christianini, Durbin,
Pevzner,watermanintrocompbio,china}.\\\\
Unfortunately, although easy to apply and define, it turns out that
the theoretical study of $L_n$ is very difficult. It is easy to see
that the global score is superadditive. This implies that  when $Z$
is taken from an ergodic process, by Kingman's subadditive ergodic
theorem, there exists a constant $\gamma^{*}$ such that
\begin{align}\label{eq:ConvLn}
\frac{L_{n}}{n}\rightarrow\gamma^{*}\quad\text{a.}\,\text{s.}\,\,\,\text{and in }\,L_{1},\quad\text{as }\,n\rightarrow\infty.
\end{align}
In the LCS case, the existence of $\gamma^{*}$ was first shown by
Chv\'{a}tal and Sankoff~\cite{Sankoff1}, but its exact value (or an
expression for it), although well estimated, remains unknown even
for i.i.d. Bernoulli sequences. Alexander~\cite{Alexander1}
established the rate of convergence of the left hand side of equation~\eqref{eq:ConvLn} in the LCS
case, a result extended by Lember, Matzinger and
Torres~\cite{LemberMatzingerTorres:2012} to general
scoring functions.\\\\
In their leading paper \cite{Sankoff1},  Chv\'{a}tal and Sankoff
first studied the the asymptotic order of $\text{Var}(\ell_n)$  and
based on some simulations, they conjectured that
$\text{Var}(\ell_{n})=o(n^{2/3})$, for $X$ and $Y$ independent
i.i.d. symmetric Bernoulli. In the case of  independent i.i.d.
sequences, it follows from Efron-Stein inequality (see, e.g.
\cite{BoucheronLugosiMassart:2013}) that
\begin{align}\label{eq:UpperBoundLn}
\text{Var}(L_{n})\leq C_{2}\,n,\quad\text{for all }\,n\in\mathbb{N},
\end{align}
where $C_{2}>0$ is an universal constant, independent of $n$. For
the LCS case, this result was proved by Steele~\cite{Steele86}.
In~\cite{Waterman:1994}, Waterman asked whether or not the linear
bound on the variance can be improved, at least in the LCS case. His
simulations showed that, in some special cases (including the LCS
case), $\text{Var}(L_{n})$ should grow linearly in $n$. These
simulations suggest the linear lower bound $\text{Var}(\ell_{n})\geq
c \cdot n$, which would invalidate the conjecture of Chv\'{a}tal and
Sankoff. In the past ten years, the asymptotic behavior of
$\text{Var}(\ell_{n})$  has been investigated by Bonetto, Durringer,
Houdr\'{e}, Lember, Matzinger and Torres, under various choices of
independent sequences $X$ and $Y$ (cf.~\cite{BonettoMatzinger:2006},
\cite{DurringerLemberMatzinger:2007}, \cite{HoudreMatzinger:2007},
\cite{LemberMatzinger:2009},
\cite{LemberMatzingerTorres:2012(2)},\cite{LemberMatzingerTorres:2013}
\cite{HoudreMa:2014}, \cite{HoudreLember} etc). In particular, in
\cite{HoudreLember} and \cite{LemberMatzingerTorres:2012(2)} a
general approach for obtaining the lower bound for moments
$E\Phi(L_n-EL_n|)$, where $\Phi$ is a convex increasing function we
worked out. For more detailed history of the problem as well as the
connection between the rate of central absolute moments of $L_n$ and
the central limit theorem $\sqrt{n}(L_n-EL_n)$, we refer to
\cite{HoudreLember}.
\\\\
 In this paper, we follow the general approach developed in \cite{LemberMatzingerTorres:2012(2)} and
 \cite{HoudreLember}, but unlike all
previous papers, we apply it for sequences that are not necessarily
independent and i.i.d. Indeed, in the present paper, we assume that
$Z$ consists of $n$ observations of aperiodic stationary Markov
chain.  Following W. Pieczynski, we call such a model as {\it
pairwise Markov chain} (PMC)
\cite{Pieczynski:2003,DerrodePieczynski:2004,DerrodePieczynski:2013}.
It is important to realize that the Markov property of $Z$ does not
imply the one of marginal processes $X$ and $Y$. On the other hand,
it is not hard to see that conditionally on $X$, the $Y$ process is
a Markov chain and, obviously, vice versa  \cite{Pieczynski:2003} .
Hence the name -- pairwise Markov chains. Thus,  we do not assume
that $X$ and $Y$ are both Markov chains, although often this is the
case. So, our model is actually rather general one including as a
special case hidden Markov models (HMM's), Markov switching models,
HMM's with dependent noise \cite{DerrodePieczynski:2013} and also
the important case where $X$ and $Y$ are independent Markov chains
or even i.i.d.. Except \cite{LemberMatzingerTorres:2012(2)}, where
among  others also  some specific independent non-i.i.d. sequences
were considered, all previous articles cited above  deal
with the case when $X$ and $Y$ are independent i.i.d. sequences.\\\\
The paper is organized as follows. In Section~\ref{gen}, we state a
very general theorem -- Theorem~\ref{general} --  for obtaining the
lower bound  of $E\Phi(|L_n-EL_n|)$ for any model (Theorem
\ref{general}). The proof of Theorem~\ref{general} is a
generalization of Theorem 3.2 in \cite{HoudreLember} and
therefore, we prove it in the appendix. Theorem \ref{general} does
not assume any particular stochastic model $Z$, instead it requires
a specific {\it random transformation} $\mathcal{ R}$ and random vectors
$U,V$ so that several general assumptions listed as {\bf A1} -- {\bf
A4} are satisfied. The objective of the present paper is to show
that under PMC-model, a random transformation $\mathcal{ R}$ as well as
$U$ and $V$ can be constructed so that the assumptions {\bf A2} --
{\bf A4} hold. In this paper, we do not formally prove the
assumption {\bf A1} -- the proof of that assumption is rather
technical and beyond the scope of the present paper. Instead of the
proof, we present some heuristic reasoning and computer simulations
to convince the reader that it holds in many cases. The computer
simulations also allow us to estimate constants in the lower
bound.\\
Finally, in Section\ref{upper} we present an upper bound of
$E|L_n-EL_n|^r$. The upper bound shows that the order of convergence
cannot be improved, so that $E|L_n-EL_n|^r\asymp n^{r\over 2}$.
\section{General lower bound in two-step approach}\label{gen}
In this  paper, we follow so-called {\it two-step} approach. This
approach has  actually been used in the most of the papers for
obtaining the lower bound of variance, but  formalized in
\cite{LemberMatzingerTorres:2012(2),HoudreLember}.\\\\
 Generally speaking, the {\it first step} is to find a random mapping independent of $Z$, usually called by us  as {\it random transformation},
$$\mathcal{R}: \Z_n \to \Z_n$$
such that for an universal constant $\e_o>0$, the following
convergence holds:
\begin{equation}\label{main}
P\big(E[L(\mathcal{R}(Z))-L(Z)|Z]\geq \e_o\big)\to 1
\end{equation}
Here, abusing a bit of the notation, $E$ denotes the expectation
over the randomness involved in $\mathcal{R}$ and $P$ denotes the
law of $Z$. \\\\
The {\it second step} is to show that equation~\eqref{main} implies the
optimal rate of convergence of absolute moments
$$\Phi\big(|L_n-EL_n|\big),$$
where $\Phi : \mathbb{R}^+ \to \mathbb{R}^+$ is a convex
non-decreasing function. To do so, we look for $U:=\mathfrak{u}(Z)$
and $V:=\mathfrak{v}(Z)$ two new random vectors (functions of $Z$),
such that $\Phi\big(|L_n-EL_n|\big)$ can be somehow estimated from
below by $\Phi(U)$. So, the control of fluctuations of $L_n$ is
reduced to the (easier) control of fluctuations of $U$ and $V$. The
current paper deals with the second step - the goal is to provide a
general theorem and to apply it for pairwise Markov chains.
\subsection{The main theorem}
Consider two given functions
$$\mathfrak{u}:\quad \mathcal{ Z}_n \to \mathbb{Z},\quad \mathfrak{v}:\quad \mathcal{ Z}_n \to
\mathbb{Z}^d$$ and define $U:=\mathfrak{u}(Z)$ (resp.
$V:=\mathfrak{v}(Z)$ ) an integer value random variable (resp.
vector). Denote by $\S$, $\S^U$ and $\S^V$ the support of
distributions of $(U,V)$, $U$ and $V$, respectively. Hence
$\S\subset \mathbb{Z}^{d+1}$, $\S^U\subset \mathbb{Z}$ and
$\S^V\subset \mathbb{Z}^{d}$. For every $v\in \S^V$, we define the
fiber of $\S^U$ as follows
$$\S(v):=\{u\in \S^U: (u,v)\in \S\}.$$
For any $(u,v)\in \S$, let
$$l(u,v):=E[L(Z)|U=u,V=v].$$
For any $(u,v)\in \S$, let $P_{(u,v)}$ denote the law of of
$Z=(X,Y)$ given $U=u$ and $V=v$, namely
$$P_{(u,v)}(z)=P(Z=z|U=u,V=v).$$
\newparagraph{Assumptions.}
The choice of the random transformation $\mathcal{R}$ and $U,V$ are
linked together through the following assumptions :
\begin{description}
  \item[A1] There exist universal constant $\e_o>0$ and a sequence
  $\Delta_n\to 0$ such that
\begin{equation*}
P\big(E[L(\tZ)-L(Z)|Z]\geq \e_o\big)\geq 1-\Delta_n.
\end{equation*}
  \item[A2] There exists an universal constant (independent of $n$) $A<\infty$ such that $L(\mathcal{R}(Z))-L(Z)\geq -A$.
  \item[A3] There exists sets $\V\subset \S^V$ and $$\U(v):=\{u_n(v)+1,u_n(v)+2,\ldots u_n(v)+m_n(v)\}\subset \S(v)$$ 
  such that for any $(u,v)$ such that $v\in \V$ and  $u\in \U(v)$,
   the following implication holds:
\begin{equation*}
\text{If  }Z\sim P_{(u,v)},\text{   then   }\mathcal{R}(Z) \sim
P_{(u+1,v)}.
\end{equation*}
  \item [A4] There exists $n_1>0$ and a function $c(v)>0$ (independent of $n$)
  such that for every $n \ge n_1$ and for every $v\in \V$,
   $$m_n(v)\geq
  c(v)\varphi_v(n)^{-1},$$ where  $\varphi_v(n)> 0$ satisfies
\begin{equation}\label{varphi}
  \min_{u\in \U(v)} P(U=u|V=v)\geq \varphi_v(n).\end{equation}
\end{description}
We note that \textbf{A4} is equivalent to the existence of
$n_1>0$ and a function $c(v)>0$ (independent of $n$)
  such that for every $n \ge n_1$ and for every $v\in \V$,
   $$ m_n(v) \cdot \min_{u\in \U(v)} P(U=u|V=v) \geq
  c(v)$$
  and $\varphi_v$ can be chosen as any function satisfying
\begin{equation}\label{mv}
  \varphi_v(n) \in [c(v)/m_n(v), \min_{u\in \U(v)} P(U=u|V=v)].\end{equation}
In what follows, we are interested in taking $\varphi_v(n)$ as small
as possible, because the smaller $\varphi_v(n)$, the bigger the lower
bound of $E\Phi\big(|L(Z)-\mu_n|\big)$ (see equation~\eqref{low1} in the theorem
below). By equation~\eqref{mv}, small $\varphi_v(n)$ means big $m_n(v)$, but
too big a set $\U(v)$ typically means an exponential decay in probability
and then the constant $c(v)$ might not exist. Therefore {\bf A4} ties
the lower bound of $E\Phi\big(|L(Z)-\mu_n|\big)$ with the size of
$\U(v)$. Clearly  $\varphi_v(n)$ can be chosen in such a way that
$\varphi_v(n) \to 0$ as $n \to \infty$ if and only if $m_n(v) \to
\infty$ as $n \to \infty$.
\begin{theorem}\label{general} Let $\Phi : \mathbb{R}^+ \to \mathbb{R}^+$ be convex
non-decreasing function and let $\mu_n$ be a sequence of reals.
Assume {\bf A1}, {\bf A2}, {\bf A3}, {\bf A4}. Suppose that
$c:=c(v)$ is independent of $v\in
\V$ and that $\varphi(n):=\sup_{v \in \V}(\varphi_v(n))\to 0$  as $n \to \infty$. 
If, in addition, there exists $b_o>0$ such that $P(V\in \V)\geq
b_o$ for any $n$ big enough, then
given any  constant $c_o \in(0, b_o c/8)$
for every sufficiently large $n$ 
\begin{equation}\label{low1}
E\Phi\big(|L(Z)-\mu_n|\big) \geq  \Phi\Big({\e_o c\over 16
\varphi(n)}\Big)c_o.
\end{equation}
\end{theorem}
\section{Pairwise Markov chains}
In this paper, we consider a rather general model. Let
$X_1,X_2,\ldots$ and  $Y_1,Y_2,\ldots$ be two random processes on
common state-space $\A=\{a_1,\ldots,a_k\}$ (i.e. r.-variables $X_i$
and $Y_i$ take values on $\A$) such that the 2D process
$Z_1,Z_2,\ldots$, where $Z_i=(X_i,Y_i)$ is an aperiodic stationary
 MC with  state space $\A \times \A$. Now the words
 $X=(X_1,\ldots,X_n)$ and $Y=(Y_1,\ldots,Y_n)$ are taken as the
 first $n$ elements from these sequences.
In what follows, we shall denote the elements of $\A \times \A$ by
capital letters and we denote by $\mathbb{P}=(p_{\mathbf{A} \mathbf{B}})_{\mathbf{A},\mathbf{B} \in \A \times \A}$ the transition
matrix of $Z$. By aperiodicity assumption, there exists an integer
$m$ such that $\mathbb{P}^m$ is primitive, i.e. has all strictly
positive entries.
\newparagraph{Random variable $V$.}
To construct $V$,  we fix  pairs $\mathbf{A},\mathbf{B}\in \A \times
\A$ such  that $P(Z_1=\mathbf{A},Z_3=\mathbf{B})>0$ and define $f:=I_G$, where $I_G$
stands for an indicator function and
$$G:=\{\mathbf{A}\}\times
\mathbb{A}\times \A\times \{\mathbf{B}\}.$$
 Let us now define  a Markov chain
$\xi:=\xi_1,\xi_2,\ldots$ as follows
$$\xi_1:=(Z_1,Z_2,Z_3),\quad \xi_2:=(Z_4,Z_5,Z_6),\ldots $$
Thus, the state space of $\xi$-chain is a subset $\X\subset
\mathbb{A}^6$ (not necessarily $\mathbb{A}^6$, because given the
zeros in $\mathbb{P}$, it might be that some triplets have zero
probability).  Since $Z$ is stationary, so is $\xi$; moreover the
aperiodicity of $Z$ implies that of $\xi$. The random variable $V$
is defined now as follows
$$V:=:\mathfrak{v}(Z):=\sum_{i=1}^{\lfloor {n\over
3}\rfloor}f(\xi_i)=\sum_{i=1}^{\lfloor {n\over
3}\rfloor}I_G(\xi_i).$$ Therefore,
$$EV=\sum_{i=1}^{\lfloor {n\over
3}\rfloor}P(\xi\in G)=\lfloor {n\over 3}\rfloor P(Z_1=\mathbf{A},Z_3=\mathbf{B}),$$
where the last equality follows from the stationarity.
Let $\alpha:={1\over 3} P(Z_1=\mathbf{A},Z_3=\mathbf{B}).$ Then $$EV=\lfloor {n\over
3}\rfloor 3\alpha .$$ When $n=3m$, for some integer $m$, then
$EV=\alpha n$, otherwise
$$EV:=\alpha_n n,\quad \text{where}\quad  \alpha -{3\alpha\over n} < \alpha_n \leq \alpha.$$
Let us define
$$\V=[\alpha_n n-K\sqrt{n},\alpha_n n+ K\sqrt{n}]\cap \S^V,$$
where $K$ is a constant specified later.
\newparagraph{Proving $P(V\in \V)>b_o$ with $b_o$ arbitrary close to 1.}
Let us now show that $P(V \in \V)$ is bounded away from zero for any sufficiently large $n$. Let
$b_o\in (0,1)$ be fixed. For that we use Hoeffding inequality for
Markov chain proven in \cite{MarkovHoeff}. The theorem assumes that
$\xi_1,\xi_2,\ldots $ satisfies the following condition: there
exists probability measure $Q$ on $\X$, $\lambda>0$ and integer
$r\geq 1$ such that for any state $x\in \X$
\begin{equation}\label{D}
P_x(\xi_{r+1}\in\cdot)\geq \lambda Q(\cdot)
\end{equation}
where $P_x(\cdot):=P(\cdot| \xi_1=x)$.
Recall that $Z$ is aperiodic and so is $\xi$, hence there is a $r$
such that  for all states $x,y \in \X$, it holds $P_x(\xi_r=y )>0$.
That implies that equation~\eqref{D} holds with $Q$ being uniform over $\X$
and $$\lambda=\min_{x,y}P(\xi_{1+r}=y|\xi_1=x)|\X|.$$ Then,
according to the theorem, given  a function $f: \X\to \mathbb{R}$,
$S_m:=\sum_{i=1}^mf(\xi_i)$, Hoeffding inequality is as follows: for
any $x\in \X$
\begin{equation}\label{MCHoffding}
P_x\big(S_m-ES_m>m\e\big)\leq \exp[-{\lambda^2(m\e-2{r\over
\lambda}\|f\|_\infty)^2\over 2m \|f\|_\infty^2 r^2}],\quad \text{if }m>2r(\lambda
\e)^{-1}\|f\|_\infty \end{equation}
where $\|f\|_\infty:=\sup\{|f(x)| \colon x \in \X\}$.
We take $m={\lfloor {n\over 3} \rfloor}$ (remember that $V:=S_{ \lfloor {n\over 3 }\rfloor }$) and
$f=I_G$ so that $\|f\|_\infty=1$. The inequality \eqref{MCHoffding} is now:
for every $\e>0$  and $x\in \X$
\begin{equation}\label{MCHoffding2}
P_{x}\big(V-
EV>
\lfloor { n\over 3}\rfloor \e\big)\leq
\exp[-{\lambda^2(\lfloor {n\over 3}  \rfloor  \e-{2r\over
\lambda})^2\over 2{n \over 3}r^2}],\quad \text{if }n>6r(\lambda
\e)^{-1}+3.\end{equation}
Take now $K$ so big that
$$\exp[-{3\over 8}\big({\lambda\over r}\big)^2K^2]<\frac{1-b_o}{2}.$$
Take now $\e=K{3\over \sqrt{n}}$; then
$$K\sqrt{n}-{3K\over \sqrt{n}} \leq \lfloor { n\over 3}\rfloor  \e \leq
K\sqrt{n}.$$
If $n$ is so big that
$${K\sqrt{n}\over 2}>{3K\over \sqrt{n}}+{2 r\over \lambda},$$
then $n>6r(\lambda
\e)^{-1}+3$ and inequality \eqref{MCHoffding2} implies
\begin{align*}
P_{x}\big(V-EV>K{\sqrt{n}}\big)&\leq \exp[-{3 \lambda^2({K\sqrt{n}}-
{3K\over \sqrt{n}}   -{2 r\over \lambda})^2\over 2 r^2 n }]\leq
\exp[-{3 \lambda^2({1\over 2}{K\sqrt{n}})^2\over 2 r^2 n }]=
\exp[-{3\over 8}\big({\lambda K\over r}\big)^2]\leq {1-b_o\over 2}.\end{align*}
Since the left hand side  holds for any initial probability distribution of $\xi$
and, then also, for any initial probability distribution of $Z$. Thus, we have
shown that there exists $n_1$ so that for every $n>n_1$
$$P(V-EV\leq K\sqrt{n})=P(V\leq \alpha n+ K\sqrt{n})>{1\over 2}+{b_o\over 2}.$$
Applying the same argument for $f=-I_G$, we obtain that
$$P(V+ EV\geq - K\sqrt{n})=P(V\geq \alpha n - K\sqrt{n})>{1\over 2}+{b_o\over 2}.$$
These two inequalities together give
$$P(\alpha_n n - K\sqrt{n} \leq V\leq \alpha_n n+ K\sqrt{n})=P(V\in
\V)>{b_o},\quad \forall n>n_1.$$
\newparagraph{Random variable $U$.} Let us now define $U$. To this aim,  fix a letter in $\A\times \A$ and let us
call it $\mathbf{D}$. The random variable $U$ is the number of states $\mathbf{D}$
in the middle of the $(\mathbf{A}\cdot \mathbf{B})$-triplets. For the formal
definition let, for any $z\in (\A\times \A)^n$,
$$\mathfrak{n}_i(z):=f(z_{3(i-1)+1},z_{3(i-1)+2},z_{3(i-1)+3}),\quad i=1,\ldots,\lfloor {n\over 3}\rfloor$$
and let us denote $(\mathfrak{n}_1(z), \ldots, \mathfrak{n}_{\lfloor n/3 \rfloor}(z))$ by $\mathfrak{n}(z)$.
Hence
$$\mathfrak{v}(z)=\sum_{i=1}^{\lfloor {n\over
3}\rfloor}\mathfrak{n}_i(z).$$ The function  $\mathfrak{u}(z)$ and random variable $U$ are defined as follows
$$\mathfrak{u}(z)=\sum_{i=1}^{\lfloor {n\over
3}\rfloor}\mathfrak{n}_i(z)I_{\mathbf{D}}(z_{3(i-1)+2}),\quad U=\mathfrak{u}(Z).$$
Clearly $\S(v)=\{0,1,\ldots, v\}$. Moreover, by the Markov property
given $V=v$, $U\sim B(v,q)$, i.e. the random variable $U$ has
binomial distribution with parameters $v$ and
$$q:={p_{\mathbf{A}\mathbf{D}}p_{\mathbf{D}\mathbf{B}}\over \sum_{\mathbf{D}'\in \A\times\A}p_{\mathbf{A}\mathbf{D}'}p_{\mathbf{D}'\mathbf{B}}}.$$
The letter $\mathbf{D}$ is chosen in such a way that $q>0$. Take now, for any $v\in
\S^V$
$$\U(v):=[vq-\sqrt{v},vq+\sqrt{v}]\cap \S(v).$$
When $v$ is big enough, then
$$\U(v)=[vq-\sqrt{v},vq+\sqrt{v}]\cap \mathbb{Z}.$$
In this case the interval contains at most $\lfloor 2\sqrt{v}+1 \rfloor$ integers.
\newparagraph{Proving {\bf A4} for $c(v)$ and $\varphi_v(n)$ independent of
$v$.}
\begin{lemma}\label{binomlclt}
Let $X\sim B(m,p)$ be a binomial   random variable  with parameters
$m$ and $p$. Then, for any constant $\beta
>0$, there exists $b(\beta, p)$ and $m_o(\beta, p)$ such that for every $b \ge b(\beta, p)$,
$m>m_o$ and $$ i \in [mp-\beta \sqrt{m},mp+\beta \sqrt{m}],$$ we have
\begin{align}\label{binom}
P(X=i)&={m \choose i }p^i(1-p)^{m-i}\geq {1\over b\sqrt{ m}}.
\end{align}
\end{lemma}
It can be shown (see \cite[equation~(4.11)]{HoudreLember}) that the constant
$b(\beta,p)$ can be taken as
$$b(\beta,p):=\sqrt{2\pi p(1-p)}\exp\Big[{\beta^2\over 2p(1-p)}\Big].$$
\\\\
From this lemma, it follows that there exists universal constant
$b(q)>0$ and  integer $v_o$ so big that for every $u\in \U(v)$,
\begin{equation}\label{vorrb}
P(U=u|V=v)\geq {1\over b\sqrt{v}},\quad v>v_o\end{equation}
given any constant $b$ satisfying
\begin{equation}\label{eq:bq}
b>b(q):=\sqrt{2\pi  q(1-q)}\exp\Big[{1\over
2q(1-q)}\Big].
\end{equation}
Recall the definition of $\V$.
There exists  $n_2>n_1$ large enough such that if $n>n_2$, then $\alpha_n
n-K\sqrt{n}>v_o$ and 
$\alpha_n n + K\sqrt{n}\leq {n\over 3}<n$. Therefore,
if $n>n_2$ then
$$\V = [\alpha_n n-K\sqrt{n},\alpha_n n+ K\sqrt{n}]\cap \mathbb{Z},$$
 every $v\in \V$ is smaller than $n$ and equation~\eqref{vorrb} holds.
Hence, when $n>n_2$, we have
$$P(U=u|V=v)\geq {1\over b\sqrt{v}}\geq {1\over b\sqrt{n}}, \quad \forall v\in \V \quad
\forall n>n_2.$$ Thus equation~\eqref{varphi} holds with
$$\varphi_v(n):=:\varphi(n):={1\over b \sqrt{n}}.$$
We can find $n_3>n_2$ large enough such that
$$\sqrt{\alpha_n n-K\sqrt{n}}\geq \sqrt{{\alpha
\over 2}n}+\frac12.$$ Therefore, if $n\geq n_3$, then every $v\in \V$
satisfies  $\sqrt{v}\geq \sqrt{{\alpha n /2}}+1/2$. Since
the minimum number of integers in the interval $\U(v)$ is $\lfloor 2\sqrt{v} \rfloor$, we
obtain 
the following inequality
$$ m_n(v) > 
2\sqrt{v}-1\geq 2\sqrt{{\alpha \over 2}}\sqrt{n}=b^{-1}\sqrt{2{\alpha}}\varphi(n)^{-1}, \qquad \forall v\in \V.$$
Thus {\bf A4} with $c=b^{-1}\sqrt{2{\alpha}}$ holds. Therefore,
$${\e_o c\over 16 \varphi(n)}={\e_o \sqrt{2\alpha n}\over 16}.$$
Thus the right hand side of equation~\eqref{low1} is
$$c_o\Phi\Big({\e_o \sqrt{2\alpha n}\over 16}\Big),$$
where for $n$ big enough
$$c={\sqrt{2\alpha}\over b},\quad c_o\leq{c\over 8}\big(b_o-\sqrt{\Delta_n}\big)$$
Since $b_o$ could be taken arbitrary close to one and $\Delta_n\to
0$, we can take any $c_o<{c\over 8}$.
\newparagraph{The random transformation $\mathcal{ R}$ and the assumption {\bf A3}.}
The random transformation  $\mathcal{ R}$ picks a random $(\mathbf{A}\cdot
\mathbf{B})$-triplet which does not have a letter $\mathbf{D}$ in-between (with
uniform distribution) and changes the letter in the middle of the
triplet  into a $\mathbf{D}$-letter.
Let $\{i_1(z), \ldots, i_{\mathfrak{v}(z)}(z)\}$ be the set of indexes corresponding
to $1$s in $ \mathfrak{n}(z)$ and define
$\mathfrak{b}(z):=(\mathfrak{b}_1(z),\ldots,\mathfrak{b}_{\mathfrak{v}(z)}(z))$ where
$$\mathfrak{b}_j(z):=I_{\mathbf{D}}(z_{3(i_j(z)-1)+2}), \quad  j=1,\ldots, \mathfrak{v}(z).
$$
With this notation, the number of $\mathbf{D}$-letters in-between the
triplets is
$$\mathfrak{u}(z)=\sum_{j=1}^{\mathfrak{v}(z)}\mathfrak{b}_j(z).$$
The random transformation $\mathcal{ R}$ acts on the set of
sequences $z$ satisfying the following condition:
$\mathfrak{u}(z)<\mathfrak{v}(z)$. Given such  a sequence,
${\mathcal R}$ picks a random zero out of
$\mathfrak{v}(z)-\mathfrak{u}(z)$ zeros in the vector
$\mathfrak{b}(z)$ (uniform distribution). Suppose that the chosen
zero is the $k$-th element of $\mathfrak{b}(z)$. Then
$z_{3(i_k(z)-1)+2}\ne \mathbf{D}$  and $\mathcal{ R}$ changes that letter into $\mathbf{D}$.
Thus $\mathcal{ R}(z)$ is a sequence such that $\mathfrak{n}_i(\mathcal{
R}(z))=\mathfrak{n}_i(z)$ for every $i=1,\ldots, \lfloor {n\over
3}\rfloor$, thus $\mathfrak{v}(\mathcal{ R}(z))=\mathfrak{v}(z)$; but
$\mathfrak{u}(\mathcal{ R}(z))=\mathfrak{u}(z)+1.$
\\\\
The following is an auxiliary and almost trivial result which we prove for the sake of
completeness.
\begin{proposition}\label{prop}
Let $Z:=(Z_1,\ldots,Z_m)$ be a vector of iid Bernoulli random variables of parameter $p$ and let
$P_{(u)}$ be the law of $Z$ given $U:=\sum^m_{i=1}Z_i=u$. Let $W\sim
P_{(u)}$, where $u<m$. Then choose a random $0$ in $W$ with uniform
distribution and change it into  one. Let $\widetilde{W}$ be the
resulting random variable. Then $\widetilde{W}\sim
P_{(u+1)}$.\end{proposition}
\begin{proof}
For any $u=0,\,\cdots,m$, let $\mathcal{A}(u)\subseteq\{0,1\}^{m}$
consist of all binary sequences containing exactly $u$ ones. For any
$z\in\mathcal{A}(u)$,
\begin{align*}
{P}\left(\left.Z=z\,\right|U=u\right)=\frac{{P}\left(Z=z,\,U=u\right)}{{P}\left(U=u\right)}=
\frac{{P}\left(Z=z\right)}{{P}\left(U=u\right)}=\frac{p^{u}(1-p)^{m-u}}{\displaystyle{\binom{m}{u}}p^{u}(1-p)^{m-u}}=
\binom{m}{u}^{-1}.
\end{align*}
In other words, ${P}_{(u)}$ is the uniform distribution on
$\mathcal{A}(u)$. Now for any $u=0,\cdots,m-1$, let $W$ be any
random vector such that $W\sim {P}_{(u)}$, then
$\widetilde{W}$ is supported on $\mathcal{A}(u+1)$. For any
$z\in\mathcal{A}(u+1)$, let $0\leq i_{1}<\cdots<i_{u+1}\leq m$ be
the positions of ones in $z$, and let $\hat{z}_{i_{j}}$,
$j=1,\cdots,u+1$, be the sequence in $\mathcal{A}(u)$ obtained from $z$ by replacing
the symbol $1$ at position $i_{j}$ with $0$. We have
\begin{align*} {P}\!\left(\widetilde{W}\!=\!z\right)\!=\!\sum_{j=1}^{u+1}{P}\left(\left.\widetilde{W}\!=\!z\,\right|W\!=\!
\hat{z}_{i_{j}}\right){P}\!\left(W\!=\!\hat{z}_{i_{j}}\right)=(u\!+\!1)\frac{1}{m\!-\!u}\binom{m}{u}^{-1}\!\!=
\binom{m}{u+1}^{-1}\!\!={P}_{(u+1)}(z).
\end{align*}
\end{proof}
Let us consider the  sequence $Z=Z_1,\ldots,Z_n$ and let
$$m:=\lfloor {n\over 3}\rfloor.$$
Recall that
$$V=\sum_{j=1}^m f(\xi_i)=\sum_{i=1}^m\eta_i,$$
where $$\eta_i:=f(\xi_i)=\mathfrak{n}_i(Z)\in \{0,1\}.$$ Since $Z$
is stationary, we have that the sequence
$\eta:=(\eta_1,\ldots,\eta_m)$ is a stationary binary sequence. As
in the proof of Proposition~\ref{prop}, let $\mathcal{ A}(v)$ be the set of binary
sequences of length $m$ having $v$ ones. It is easy to see that
additional conditioning on $U$ will not change the conditional
probability of $\eta$ (given $u\leq v$), because for any vector
$a\in \mathcal{ A}(v)$ we have $\{\eta=a\} \subseteq \{V=v\}$ and
\begin{align*}
P(\eta=a|V=v,U=u)&={P(U=u,V=v,\eta=a)\over
P(U=u,V=v)}={P(U=u|\eta=a)P(\eta=a)\over
P(U=u,V=v)}\\
&={\binom{v}{u}q^u(1-q)^{v-u}P(\eta=a)\over P(U=u,V=v)}.\end{align*} Since
$$P(U=u,V=v)=\sum_{a\in \mathcal{ A}(v)}P(U=u|\eta=a)P(\eta=a)=\binom{v}{u}q^u(1-q)^{v-u}P(V=v),$$
we have
\begin{equation}\label{uv}
P(\eta=a|V=v,U=u)={P(\eta=a,V=v)\over P(V=v)}=P(\eta=a|V=v).\end{equation}
For any $u\leq v\leq m$, let $\mathcal{ B}(u,v)$ be the set of sequences  such
that the value of $\mathfrak{u}$ and $\mathfrak{v}$ are $u$ and $v$ respectively,
that is
$$\mathcal{ B}(u,v)=\{z\in (\A\times \A)^n,\quad \mathfrak{u}(z)=u,\quad
\mathfrak{v}(z)=v\}.$$ Fix $u\leq v\leq m$ and  $Z_{(u,v)}\sim
P_{(u,v)}$ (i.e.~$P(Z_{(u,v)}=z)=P(Z=z|U=u, V=v)$). Let us compute $P_{(u,v)}$. To this aim
define  $B:=(B_1,\ldots,B_V)\equiv \mathfrak{b}(Z)$. Now for any $z\in \mathcal{
B}(u,v)$,
by definition of $P_{(u,v)}$,
since $\{Z=z\} \subseteq \{\eta=\mathfrak{n}(z), B=\mathfrak{b}(z)\} \subseteq \{U=u, V=v\}$, we have
\begin{align*}
P({Z_{(u,v)}}=z)&= P(Z=z|U=u,V=v)=P({Z}=z,\eta = \mathfrak{n}(z),B=\mathfrak{b}(z)| U=u,V=v)\\
&=P({Z}=z|\eta =
\mathfrak{n}(z),B=\mathfrak{b}(z))P(\eta = \mathfrak{n}(z),B=\mathfrak{b}(z)|U=u,V=v).
\end{align*}
Given $\eta$, let  $Z'$ be the random vector obtained by collecting all
random variables from $(Z_1,\ldots,Z_n)$ corresponding to the
triplets where $\eta_i=0$. And, analogously, let $z'$ be the
vector obtained by $z$ by collecting the triplets where $\mathfrak{n}_i(z)=0$. From the
Markov property we have
\begin{align*}
P(Z'=z'|\eta = \mathfrak{n}(z), B=\mathfrak{b}(z))=P(Z'=z'|\eta =
\mathfrak{n}(z)).\end{align*} Let $1\leq i_1<\cdots < i_v\leq m$ be
the indexes of corresponding ones in $\mathfrak{n}(z)$. Then, from
$\mathfrak{b}(z)$ we know for every $j=1,\ldots,v$, whether
$z_{3(i_j-1)+2}$ equals $\mathbf{D}$ or not. But this does not fully
determine the values of $z_{3(i_j-1)+2}$. Hence
\begin{align*}
&P({Z}=z|\eta = \mathfrak{n}(z),B=\mathfrak{b}(z))=\\
&P(Z'=z'|\eta =
\mathfrak{n}(z))\prod_{j=1}^vP(Z_{3(i_j-1)+2}=z_{3(i_j-1)+2} |
Z_{3(i_j-1)+1}=\mathbf{A},Z_{3(i_j-1)+3}=\mathbf{B},B_j=\mathfrak{b}_j(z)).\end{align*}
If, in the product above, $\mathfrak{b}_j(z)=1$, then
$z_{3(i_j-1)+2}=\mathbf{D}$ and $$P(Z_{3(i_j-1)+2}=z_{3(i_j-1)+2} |
Z_{3(i_j-1)+1}=\mathbf{A},Z_{3(i_j-1)+3}=\mathbf{B},B_j=\mathfrak{b}_j(z))=1,$$
otherwise
\begin{align*}
&P(Z_{3(i_j-1)+2}=z_{3(i_j-1)+2} |
Z_{3(i_j-1)+1}=\mathbf{A},Z_{3(i_j-1)+3}=\mathbf{B},B_j=0)=\\
& P(Z_{3(i_j-1)+2}=z_{3(i_j-1)+2} |
Z_{3(i_j-1)+1}=\mathbf{A},Z_{3(i_j-1)+3}=\mathbf{B},Z_{3(i_j-1)+2}\ne
\mathbf{D})=:\rho({3(i_j-1)+2}, z);\end{align*}
note that
\begin{equation}\label{eq:sum1}
 \sum_{F \in \mathbb{A}\times\mathbb{A} \colon F \not = \mathbf{D}} P(Z_{3(i_j-1)+2}= F |
Z_{3(i_j-1)+1}=\mathbf{A},Z_{3(i_j-1)+3}=\mathbf{B},Z_{3(i_j-1)+2}\ne
\mathbf{D})=1.
\end{equation}
Thus
\begin{align*}
\prod_{j=1}^vP(Z_{3(i_j-1)+2} &=z_{3(i_j-1)+2}|
Z_{3(i_j-1)+1}=\mathbf{A},Z_{3(i_j-1)+3}=\mathbf{B},B_j=\mathfrak{b}_j(z))\\
&=\mathop{\prod_{j=1, \ldots,v}}_{\mathfrak{b}_j(z)=0}P(Z_{3(i_j-1)+2}=z_{3(i_j-1)+2} |
Z_{3(i_j-1)+1}=\mathbf{A},Z_{3(i_j-1)+3}=\mathbf{B},Z_{3(i_j-1)+2}\ne \mathbf{D})\\
&=
\mathop{\prod_{j=1, \ldots,v}}_{\mathfrak{b}_j(z)=0}\rho({3(i_j-1)+2},z)=:\rho_z.\end{align*}
From equation~\eqref{uv}, we know
$$P(\eta = \mathfrak{n}(z)|U=u,V=v
)=P(\eta = \mathfrak{n}(z)|V=v
).$$ By the Markov property
$$P(B=\mathfrak{b}(z)|\eta=\mathfrak{n}(z),U=u,V=v)=P(B=\mathfrak{b}(z)|\eta=\mathfrak{n}(z),
U=u)$$ and this probability is equal to  the probability that $v$ i.i.d
Bernoulli random variables take values $\mathfrak{b}(z)$ given their
sum is equal to $u$. This probability is $\binom{v}{u}^{-1}$. Thus,
$$P(B=\mathfrak{b}(z)|\eta=\mathfrak{n}(z),U=u,V=v)=\binom{v}{u}^{-1}.$$
Therefore, for any $z\in \mathcal{ B}(u,v)$, we have
\begin{equation}\label{eq:1}
P({Z_{(u,v)}}=z)=P(Z'=z'|\eta =
\mathfrak{n}(z))\rho_z P(\eta = \mathfrak{n}(z)|V=v)\binom{v}{u}^{-1}.
\end{equation}
We apply now the random transformation and we compute
$P(\mathcal{ R}(Z_{(u,v)})=z)$.
Clearly, given $z \in \mathcal{B}(u+1,v)$,
\[
 P(\mathcal{ R}(Z_{(u,v)})=z)= \sum_{\tilde z \in \mathcal{B}(u,v)} P(\mathcal{ R}(Z_{(u,v)}=z| Z_{(u,v)}=\tilde z) P(Z_{(u,v)}=\tilde z)=(*)
\]
and
\[
 P(\mathcal{ R}(Z_{(u,v)})=z| Z_{(u,v)}=\tilde z)=
 \begin{cases}
0 & \textrm{if }\tilde z \not \in H_\mathcal{ R}(z)\\
1/(v-u)  & \textrm{if }\tilde z \in H_\mathcal{ R}(z)\\
 \end{cases}
\]
where
$$H_\mathcal{ R}(z):=\{\tilde z \colon P(\mathcal{ R}(\tilde z)=z)>0\}=
\mathop{\bigcup_{j=1, \ldots,v}}_{\mathfrak{b}_j(z)=1}
\{\tilde z \colon P(\mathcal{ R}(\tilde z)=z)>0, \tilde z_{3(i_j-1)+2)} \not = \mathbf{D}\}$$
the latter being the union of $u+1$ pairwise disjoint sets. 
Define $\widetilde \eta:=\mathfrak{n}(\mathcal{ R}(Z_{(u,v)}))$ and observe that
if $\tilde z \in H_\mathcal{ R}(z)$ then
$P(\mathcal{ R}(Z_{(u,v)})^\prime=z^\prime|\widetilde \eta =\mathfrak{n}(z) )=
P(Z^\prime=\tilde z^\prime|\eta =\mathfrak{n}(\tilde z) )$ and
$P(\widetilde \eta =\mathfrak{n}(z)|V=v)=P(\eta =\mathfrak{n}(\tilde z)|V=v)$.
Moreover $\sum_{\tilde z \in H_\mathcal{ R}(z)} \rho_{\tilde z}=(u+1)\rho_z$ (decompose the sum using the above partition
of $H_\mathcal{ R}(z)$ into $u+1$ subsets  and use equation~\eqref{eq:sum1}). Thus, by computing
$P(Z_{(u,v)}=\tilde z)$ by means of equation~\eqref{eq:1}, we obtain
\[
\begin{split}
 (*)&= \sum_{\tilde z \in H_\mathcal{ R}(z)} P(Z'=\tilde z'|\eta =
\mathfrak{n}(\tilde z))\rho_{\tilde z}
 {P(\eta =\mathfrak{n}(\tilde z)|V=v)}{\binom{v}{u}^{-1}} \frac{1}{v-u} \\
  &= \sum_{\tilde z \in H_\mathcal{ R}(z)} P(\mathcal{ R}(Z_{(u,v)})^\prime=z^\prime|\widetilde \eta =\mathfrak{n}(z) )\rho_{\tilde z}
  {P(\widetilde \eta =\mathfrak{n}(z)|V=v)}{\binom{v}{u}^{-1}} \frac{1}{v-u}\\
   &= P(\mathcal{ R}(Z_{(u,v)})^\prime=z^\prime|\widetilde \eta =\mathfrak{n}(z) )\rho_z
  {P(\widetilde \eta =\mathfrak{n}(z)|V=v)}{\binom{v}{u+1}^{-1}} \\
  \end{split}
\]
which, according to equation~\eqref{eq:1}, implies that $\mathcal{ R}(Z_{(u,v)}) \sim P_{(u+1,v)}$
and the proof is complete.
\newparagraph{Main result.}
We have defined the random transformation $\mathcal{ R}$, random
variables $U$, $V$ and sets $\U(v)$ and $\V$ such that assumptions
{\bf A3} and {\bf A4} with $\varphi(n)$ hold. Since $\mathcal{ R}$
changes at most two letters at time, by equation~\eqref{eq:MaxScoreChange},
the assumption {\bf A2} holds for $A=2 \Delta$. Thus, recalling that $b_o$ can be chosen
arbitrarily close to 1, from Theorem~\ref{general}
we have the
following result.
\begin{theorem} \label{thm:part}
Let  $\Phi : \mathbb{R}^+ \to \mathbb{R}^+$ be convex
non-decreasing function and let $\mu_n$ be a sequence of reals. If
there exists $\e_o>0$ such that the random transformation $\mathcal{ R}$
satisfies {\bf A1}, then for every $n$ sufficiently large, the following
inequality holds
\begin{equation}\label{thm2}
E\Phi\big(|L(Z)-\mu_n|\big) \geq c_o\Phi\Big({\e_o \sqrt{2\alpha
n}\over 16}\Big),\end{equation} where 
$\alpha={1\over 3} P(Z_1=\mathbf{A},Z_3=\mathbf{B})$ and $0<c_o< b(q)^{-1}\sqrt{2 \alpha}/8$
($b(q)$ defined as in equation~\eqref{eq:bq}).
\end{theorem} In particular, when $\Phi(x)=x^r$, for $r\geq 1$ and
$\mu_n=EL(Z)$, then equation~\eqref{thm2} is
$$E\mid L(Z)-EL(Z)\mid ^r\geq c_o\Big({\e_o \sqrt{2\alpha}
\over 16}\Big)^rn^{r\over 2},$$ where $$c_o<{\sqrt{2 \alpha}\over 8b(q)}.$$
Taking $r=2$, we obtain the lower bound for variance
$$\v(L(Z))\geq a_o n,\quad a_o:={2c_o\alpha \over 16^2}\e_o^2.$$
\subsection{Combining random transformations}\label{subsec:combined} Suppose ${\bf
A}_i,{\bf B}_i,{\bf D}_i$, $i=1,2$ are pairs of letters and let us
briefly consider a  random transformation $\mathcal{ R}$ that  picks
either a random $(\mathbf{A}_1\cdot \mathbf{B}_1)$-triplet which
does not have a letter $\mathbf{D}_1$ in-between    or  a random
$(\mathbf{A}_2\cdot \mathbf{B}_2)$-triplet which does not have a
letter $\mathbf{D}_2$ in-between    (with uniform distribution over
both kind of triplets) and changes the letter in the middle of the
triplet either into $\mathbf{D}_1$-letter (if the chosen triplet was
$(\mathbf{A}_1\cdot \mathbf{B}_1)$) or  into $\mathbf{D}_2$-letter
(if the chosen triplet was $(\mathbf{A}_2\cdot \mathbf{B}_2)$). Such
a transformation $\mathcal{R}$ can be considered as a combination of
two random transformations: $\mathcal{R}_1$ that acts on
$(\mathbf{A}_1\cdot \mathbf{B}_1)$-triplets and $\mathcal{R}_2$ that
acts on $(\mathbf{A}_2\cdot \mathbf{B}_2)$.
We suppose that $(\mathbf{A}_1,\mathbf{B}_1) \not = (\mathbf{A}_2, \mathbf{B}_2)$.
Thus, for $i=1,2$, we
now have the random variables $V_i$ that count $(\mathbf{A}_i\cdot
\mathbf{B}_i)$-triplets (and are dependent on each other) and
random variables $U_i$ that counts number of states ${\bf D}_i$
in-between the triplets.  Let $q_i$ be the probability of finding
a $\mathbf{D}_i$-letter inside
$(\mathbf{A}_i\cdot \mathbf{B}_i)$-triplet.
Thus given $V_i=v_i$,
$U_i\sim B(v_i,q_i)$, $i=1,2$. Given $V_1$ and $V_2$, the random
variables $U_1$ and $U_2$ are independent. \\
We are now going to  define the  combined random transformation
$\mathcal{R}$. In what follows, let $V=(V_1,V_2)$ and $U=U_1+U_2$.
Given $V=v:=(v_1,v_2)$, the random variable $U$ takes values from
$0,1,\ldots,v_1+v_2$. Now define  the probabilities
$$p(l|u,v):=P(U_1=l|U=u,V=v),\quad l=l_1,l_1+1,\ldots,l_2,$$
where  $l_1=l_1(u,v_2):=(u-v_2)\vee 0 $ and $l_2=l_2(u,v_1):=u\wedge v_1$. Thus $p(l|u,v)$
is the probability that there are $l$ $\mathbf{D}_1$-letters (inside
the triplets) given the sum of $\mathbf{D}_1$ and $\mathbf{D}_2$
letters (inside the corresponding triplets)  is $u$. The random
transformation $\mathcal{R}$ picks the side $i$ with certain
probability $r_i$ and then applies the transformation
$\mathcal{R}_i$. In order for the composed random transformation
$\mathcal{R}$ to satisfy {\bf A3}, the probabilities $r_i$ should be
chosen carefully. To this aim, given $z$, define $\mathfrak{u}_i(z)$ and
$\mathfrak{v}_i(z)$, $i=1,2$ as usual and let
$\mathfrak{w}(z):=(\mathfrak{u}_1(z),\mathfrak{u}_2(z),\mathfrak{v}_1(z),\mathfrak{v}_2(z))$.
We now define the probabilities
$r_i(z)=r_i(\mathfrak{w}(z))=r_i(u_1,u_2,v)$ such that
$r_1(z)+r_2(z)=1$, $r_1(v_1,u_2,v)=r_2(u_1,v_2,v)=0$ and the
following conditions hold:
\begin{align}\label{kesk}
&r_1(l-1,u-l+1,v)p(l-1|u,v)+r_2(l,u-l,v)p(l|u,v)=p(l|u+1,v),\quad
l_2\geq l>l_1
\\\label{aar2}
&r_2(0,u,v)p(0|u,v)=p(0|u+1,v),\text{when  } u<v_2\\\label{aar1}
&r_1(u,0,v)p(u|u,v)=p(u+1|u+1,v),\text{when  } u<v_1\end{align}
for all $u=0,\ldots, v_1+v_2-1$.
Now
for any $w:=(u_1,v_1,u_2,v_2)$, such that $ v_i\geq u_i\geq 0$
and $v_1+v_2\leq m$, we define a random variable $T_w$ such that
$P(T_w=i)=r_i(w)$, $i=1,2$ and given the random variables
$U_i=u_i,V_i=v_i$, $T_w$ is independent of $Z$. The transformation
$\mathcal{R}$ is now formally defined as follows:
$$\mathcal{R}=\left\{
             \begin{array}{ll}
               \mathcal{R}_1(z), & \hbox{if $T_{\mathfrak{w}(z)}=1$;} \\
               \mathcal{R}_2(z), & \hbox{if $T_{\mathfrak{w}(z)}=2$.}
             \end{array}
           \right.$$
In general, the probabilities $r_i$ depend on the probabilities
$q_i$. When $q_1=q_2$, then 
$$r_i(u_1,u_2,v):={{v}_i-{u}_i\over
({v}_1-{u}_1)+({v}_2-{u}_2)},\quad i=1,2$$
satisfy the requirements.
Thus, in that case
$\mathcal{R}$ just picks one $(\mathbf{A}_i\cdot
\mathbf{B}_i)$-triplet over all such triplets with no
$\mathbf{D}_i$-letter inside with {\it uniform distribution}, whilst
in the case $q_1\ne q_2$, the distribution is not uniform. It is
easy to see  that such $r_i$ satisfy conditions (\ref{kesk})
(\ref{aar1}) and (\ref{aar2}). Indeed, the reader can easily prove the
following statement.
\begin{proposition}
Let $X\sim B(v_1,q)$ and $Y\sim B(v_2,q)$ two independent binomially
distributed random variables. Then for any integers $l$ and $u$ such
that $u<v_1+v_2$ and $u \wedge v_1 \geq l> (u-v_2) \vee 0$ we have
\begin{align*}
{v_1-l+1\over v_1+v_2-u}P(X=l-1|X+Y=u)+{v_2-u+l\over
 v_1+v_2-u}P(X=l|X+Y=u)=P(X=l|X+Y=u+1).
\end{align*}
Moreover, when $u<v_2$, then $${v_2-u\over
v_1+v_2-u}P(X=0|X+Y=u)=P(X=0|X+Y=u+1).$$
\end{proposition}
Clearly $\mathcal{R}$ satisfies {\bf A2}. We now show that it also
satisfies {\bf A3}. Fix $v=(v_1,v_2)$ such that $v_1+v_2\leq m$.
Now, we can decompose the measure $P_{(u,v)}$ as follows
\begin{equation}\label{decomp}
P_{(u,v)}=\sum_{l=l_1}^{l_2} P_{(l,u-l,v)}p(l|u,v),
\end{equation}
where $P_{(l,u-l,v)}$ is the distribution of $Z$ given
$U_1=l,U=u,V=v$. We know that  $\mathcal{R}_i$ satisfies {\bf A3}
for any $u=\{0,1,\ldots,v_i-1\}$, thus the following holds: when
$Z\sim P_{(l,u-l,v)}$ and $l<v_1$, $u-l<v_2$, then
$\mathcal{R}_1(Z)\sim P_{(l+1,u-l,v)}$ and $\mathcal{R}_2(Z)\sim
P_{(l,u-l+1,v)}$. Therefore, if $Z\sim P_{(u,v)}$, then
$$\mathcal{R}(Z)\sim
\sum_{l=l_1}^{l_2}\big( P_{(l+1,u-l,v)}r_1(l,u-l,v)+
P_{(l,u-l+1,v)}r_2(l,u-l,v) \Big) p(l|u,v).$$
Thus, by
equation \eqref{kesk}
\begin{align*}
\mathcal{R}(Z)&\sim
P_{(l_1,u-l_1+1,v)}r_2(l_1,u-l_1,v)p(l_1|u,v) \\
&+\sum_{l=l_1+1}^{l_2}
P_{(l,u-l+1,v)}\Big(r_1(l-1,u-l+1,v)p(l-1|u,v)+r_2(l,u-l,v)p(l|u,v)\Big)\\
&+P_{(l_2+1,u-l_2,v)}r_1(l_2,u-l_2,v)p(l_2|u,v)
=P_{(l_1,u-l_1+1,v)}r_2(l_1,u-l_1,v)p(l_1|u,v)\\
&+\sum_{l=l_1+1}^{l_2}P_{(l,u+1-l,v)}p(l|u+1,v)+P_{(l_2+1,u-l_2,v)}r_1(l_2,u-l_2,v)p(l_2|u,v)=(*).
\end{align*} If $u<v_1$ and $u<v_2$, then  $l_1(u,v_2)=l_1(u+1,v_2)=0$ and $l_2(u+1,v_1)=u+1=l_2(u,v_1)+1$,
thus by equations \eqref{aar2} and \eqref{aar1} we
obtain that $(*)$ equals
\[
\begin{split}
P_{(0,u+1,v)}p(0|u+1,v)&+\sum_{l=1}^{u}P_{(l,u+1-l,v)}p(l|u+1,v)+P_{(u+1,0,v)}p(u+1|u+1,v)=\\
& \sum_{l=l_1(u+1,v_2)}^{l_2(u+1,v_1)} P_{(l,u+1-l,v)}p(l|u+1,v)
=P_{(u+1,v)}.\\
\end{split}
\]
 If $u\geq v_1$ and $u<v_2$, then $l_1(u,v_2)=l_1(u+1,v_2)=0$, $l_2(u,v_1)=l_2(u+1,v_1)=v_1$
and then by equation~\eqref{aar2} and since  $r_1(v_1,u-v_1,v)=0$, we
 have that
$(*)$ equals
$$P_{(0,u+1,v)}p(0|u+1,v)+\sum_{l=1}^{v_1}P_{(l,u+1-l,v)}p(l|u+1,v)=\sum_{l=l_1(u+1,v_2)}^{l_2(u+1,v_1)} P_{(l,u+1-l,v)}p(l|u+1,v)
=P_{(u+1,v)}.$$
If $u<v_1$ and $u\geq v_2$, then $l_1(u+1,v_2)=u+1-v_2=l_1(u,v_2)+1$ and $l_2(u+1,v_1)=u+1=l_2(u,v_1)+1$,
thus by equation~\eqref{aar1}
and since $r_2(u-v_2,v_2,v)=0$ we obtain that $(*)$ equals
$$\sum_{l=u-v_2+1}^{u}P_{(l,u+1-l,v)}p(l|u+1,v)+P_{(u+1,0,v)}p(u+1|u+1,v)=\sum_{l=l_1(u+1,v_2)}^{l_2(u+1,v_1)} P_{(l,u+1-l,v)}p(l|u+1,v)=P_{(u+1,v)}.$$
Finally, if  $u\geq v_1$ and $u\geq v_2$, then $l_1(u+1,v_2)=u+1-v_2=l_1(u,v_2)+1$, $l_2(u,v_1)=l_2(u+1,v_1)=v_1$.
Since $r_2(u-v_2,v_2,v)=r_1(v_1,u-v_1,v)=0$ and $(*)$ equals
$$\sum_{l=u-v_2+1}^{v_1}P_{(l,u+1-l,v)}p(l|u+1,v)=\sum_{l=l_1(u+1,v_2)}^{l_2(u+1,v_1)} P_{(l,u+1-l,v)}p(l|u+1,v)=P_{(u+1,v)}.$$
Thus, we have shown that $\mathcal{R}(Z)\sim P_{(u+1,v)}$ and {\bf
A3} is fulfilled for any $u\in \{0,1,\ldots, v_1+v_2-1\}$.\\
{To the end of the paragraph, let us skip $n$ from the notation and
let $\mathcal{V}:=\mathcal{V}_1\times \mathcal{V}_2$. For
$(v_1,v_2)\in \mathcal{V}$, let
$\mathcal{U}_i(v_i):=[v_iq_i-\sqrt{v_i},v_iq_i+\sqrt{v_i}]\cap
\mathbb{Z}$ and
$$\mathcal{U}(v):=[(v_1q_1+v_2q_2)-\sqrt{v_1}\wedge\sqrt{v_2}),(v_1q_1+v_2q_2)+\sqrt{v_1}\wedge\sqrt{v_2}]\cap
\mathbb{Z}.$$}
It is not difficult to show that for every $u \in \mathcal{U}(v)$ the cardinality of
$\{(u_1,u_2) \in \mathcal{U}_1(v_1) \times \mathcal{U}_2(v_2) \colon u_1+u_2=u\}$ is at least $\lfloor \sqrt{v_1}\wedge\sqrt{v_2} \rfloor$.
In order to show that $\mathcal{R}$ satisfies {\bf A4}, we assume
without loss of generality that $v_1 \le v_2$, whence
$\sqrt{v_1}\wedge\sqrt{v_2}=\sqrt{v_1}$. We know that
$\mathcal{R}_i$ satisfies {\bf A4}, so for $i=1,2$ there exists a
constant $b_i$ such that for every $u_i\in \mathcal{U}_i(v_i)$ and
$n$ big enough, we have
$$P(U_i=u_i|V_i=v_i)\geq {1\over b_i\sqrt{v_i}}$$
where $b_i$ depends only on $q_i$ (see Lemma~\ref{binomlclt}).
Now observe that $\lfloor x \rfloor/x \ge 1/2$ for all $x \ge 1$; thus, for every $u\in \mathcal{U}(v)$ and every sufficiently large $n$,
\begin{align*}
P(U=u|V=v)&\geq \sum_{u_i\in
\mathcal{U}_i(v_i):u_1+u_2=u}P(U_1=u_1|V_1=v_1)P(U_2=u_2|V_2=v_2)\\
&\geq
\sum_{u_i\in \mathcal{U}_i(v_i):u_1+u_2=u}{1\over
b_1b_2\sqrt{v_1v_2}}
\geq {1\over 2 b_1b_2\sqrt{v_2}}
\geq {1\over
 2b_1b_2\sqrt{n}}.
\end{align*}
The number of elements in  $\mathcal{U}(v)$ is bigger than
$2\sqrt{v_1}-1$ 
and since there exists a constant $c>0$ such that $\sqrt{v_1}\geq
c\sqrt{n}$, we see that {\bf A4} holds with
with $\varphi_v(n)$
independent of $v$.  \\
Finally, we note that from $P(V_i\in \mathcal{V}_i)\geq b_o$ (where
$b_o$ is close to 1), it follows that  $P(V\in \mathcal{V})\geq
1-2(1-b_o)$.
\subsection{About assumption {\bf A1} for the longest common subsequence}
The assumption {\bf A1} depends very much on concrete model and the
scoring function $S$. Even when {\bf A1} it is intuitively
understandable, it is, in general, very difficult to prove. Let us
briefly explain the intuition behind {\bf A1} in the case of the
longest common subsequence. Thus $L(Z)=\ell(Z)$ is the length of the
longest common subsequence.\\
Suppose that there is a letter in $\A$, say $a$ so that the pair
$\mathbf{A}^*:=(a,a)$ has high probability. Such a situation might occur in
many cases in practice, for  example when $X$ and $Y$ are
independent stationary Markov chains having the same distribution
and the probability $P(X_1=a)$ is very high. Since the pair $\mathbf{A}^*$
has high probability, typically the sequence $Z_1,Z_2,\ldots, Z_n$
has many $\mathbf{A}^*$s. Then, in the construction of $V$ and $U$, take
$\mathbf{A}=\mathbf{B}=\mathbf{D}=\mathbf{A}^*$. In this case the random variable $V$ counts the number
of $(\mathbf{A}^*\cdot \mathbf{A}^*)$ triplets in certain positions (i.e first three
letters, then letters $4,5,6$ etc) and $U$ counts the number of
$\mathbf{A}^*$s between these triplets. The random variable $\mathcal{ R}$ now
picks any non-$\mathbf{A}^*$  in-between the triplet (with uniform
distribution) and changes it into $\mathbf{A}^*$. Clearly $\mathcal{ R}$ then
changes at least one non-$a$-letter into $a$-letter. As a result,
the number of $\mathbf{A}^*$s increases and the number of $a$s in $X$ and
$Y$ increases as well. Since there are many $\mathbf{A}^*$s in $Z$ and,
therefore, many $a$s in $X$ and $Y$, any longest common
subsequence has to connect many $a$s on the $X$-side with $a$s on the
$Y$-side. If the probability of $a$ in $X$ is very high, then any
longest common subsequence consists of mostly $a$-pairs. It does
necessarily mean that two $a$s in the same position (thus a
$\mathbf{A}^*$-pair) would be necessarily connected by LCS, but it is very
likely that both $a$s in a $\mathbf{A}^*$ are connected. In fact, as the
simulations in \cite{barder} showed, with highly asymmetrical
distribution of $X_i$ (i.e. having a letter $a$ with high
probability) the subsequence that aligns as many $a$s as possible
is very close to being the longest. \\
Hence, if $X$ and $Y$ sequences both have many $a$-letters then any
LCS connects mostly $a$-letters. That implies that non-$a$-letters
have bigger likelihood to remain unconnected, because connecting a
pair of non $a$-letters will typically destroy many connected
$a$-letter pairs. Thus changing at least one non-$a$-letter  into an
$a$-letter, has tendency to increase LCS.\\\\
The above-described approach has been formalized  in
\cite{HoudreMa:2014, LemberMatzinger:2009}. In those
articles, $X$ and $Y$ are considered independent i.i.d. sequences,
where $X_1$ and $Y_1$ have the common asymmetric distribution over
$\A$ (in \cite{LemberMatzinger:2009} a two letter alphabet is
considered; in \cite{HoudreMa:2014} the result is generalized for
many letter alphabet). The asymmetry means that one letter, say $a$,
has the probability close to one. Thus both sequences consists
mostly of $a$s. In these papers, the random transformation picks
any non-$a$ letter from these letters in $X$ and $Y$ letters and
then changes it into $a$. In this case, the random variable $U$
counts $a$s in $X$ and $Y$ sequence, and there
is no need for $V$-variable, formally take $V\equiv 2n$.\\
Formally the described random transformation used in these two
papers differs from the one in the present article by several
aspects:
\begin{enumerate}
  \item The sequences $X$ and $Y$ are considered separately, not
  pairwise. This is due to the independence of $X$ and $Y$. If
  $X$ and $Y$ are independent Markov chains, then we could
  define $\mathcal{ R}$ also as follows: consider all
  (non-overlapping) triplets in $X$ and $Y$ sequences separately
  and let $V$ count the $a\cdot a$-ones. The maximal number of
  such triplets would be $2\lfloor{n\over 3} \rfloor$, not
  $\lfloor{n\over 3} \rfloor$ as in our case. Then pick any
  triplet with non-$a$-letter in between (either in $X$ or $Y$
  sequence) and change the middle letter into $a$. The random
  variable $U$ counts the $a$s in the middle of the triplets.
  Surely, due to the independence of $X$ and $Y$, the
  conditional distribution of $U$ given $V=v$ is still Binomial
  and it is straightforward to verify that everything else holds
  as well. When  $X$ and $Y$ are dependent, one need them to
  consider them pairwise in order to obtain the conditional
  independence of $B_1,\ldots,B_v$ given $V=v$.
  \item There are no fixed  $a\cdot a$-neighborhoods and hence also no $V$-variable. The fixed
  neighborhood is not needed, because $X$ and $Y$ already
  consists of independent random variables. And the number of
  $a$s is Binomially distributed. In the case on Markov chains,
  the  fixed neighborhoods are needed, again, to obtain the
  conditional independence of $B_1,\ldots,B_v$ given $V=v$.
  Without neighborhoods, there is obviously no need for
  prescribed triplet-locations.
\end{enumerate}
Thus, although formally different, the random transformation in the
present article is of the same nature as the ones used in
\cite{LemberMatzinger:2009,HoudreMa:2014}, where it is proven that
when the probability of $a$ is close to zero then assumption {\bf
A1} holds (see \cite[Theorem 2.1]{HoudreMa:2014},
\cite[Theorem 2.2]{LemberMatzinger:2009}). Therefore, it is reasonable to
believe, that in the case where an $\mathbf{A}^*=(a,a)$ pair has high enough
probability, then  $\mathcal{ R}$ that replaces a random non-$\mathbf{A}^*$ pair
by $\mathbf{A}^*$ satisfies {\bf A1}. To prove that, however, is beyond
the scope of the current paper and needs a separate article.\\\\
Suppose now that there is a pair of different letters $(a,b)$ such
that $P(Z_1=(a,b))$ is close to one. Then take  $\mathbf{A}=\mathbf{B}=\mathbf{D}=(a,b)$ and let
the random transformation to change a non $(a,b)$-pair into
$(a,b)$-pair. Clearly such a random transformation tends to decrease
the length of LCS. But when such a transformation decreases the
length of LCS by a fixed $\e_o$, then defining $L(Z)=n-\ell(Z)$, we
see that {\bf A1} still holds. In other words, it is not important
whether $\mathcal{ R}$ actually increases or decreases the score,
important is that in influences it. Hence, if there is a pair
in $\A\times \A$ occurring with sufficiently large probability, then
the approach in \cite{HoudreMa:2014,LemberMatzinger:2009} applies.
\subsection{Simulations}
The goal of the present subsection is to  check the assumption {\bf
A1} by simulations. Given random transformation $\mathcal{ R}$ and a
sequence $Z=Z_1,\ldots,Z_n$, let us denote
$$E_n:=E[L(\mathcal{ R}(Z))|Z]-L(Z),$$
where the expectation is taken over the random transformation. Under
${\bf A1}$, there exists $\epsilon_o>0$ such that
$$P(E_n\geq \epsilon_o)\to 1.$$
If the convergence above is fast enough, then $P(E_n\geq
\epsilon_o,\,\,\rm{ev})=1$ implying that $\lim\inf_n E_n\geq
\epsilon_o$, a.s.. Our objective now is to study the asymptotic
behavior of $E_n$ for several PMC-models. Throughout the subsection
the score is the length of LCS, i.e. $L(Z)=\ell(Z)$. Let us start
with the model.
\newparagraph{The model.}
Before introducing our specific model we state the following lemma
whose proof is is included for the sake of completeness.
\begin{lemma}\label{lem:partition}
 Let $Z_1,Z_2,\ldots$ be a Markov chain  on $X$ with transition matrix $\mathbb{P}=(p_{xy})_{x,y \in X}$.
  Suppose that $\{A_i\}_{i \in I}$
 is a partition of $X$ and
 define $\pi:X\rightarrow I$ as $\pi(x)=i$ if and only if $x \in A_i$.
 Then the following assertions are equivalent:
 \begin{enumerate}
  \item for every initial distribution of $Z_0$, $\pi(Z_1),\pi(Z_2),\ldots$
  is a Markov chain on $I$ with transition matrix $\mathbb{Q}:=(q_{ij})_{i,j \in I}$;
  \item for all $i,j \in I$, $x \in A_i$
  \begin{equation}\label{eq:constant}
    \sum_{y \in A_j} p_{xy} = q_{ij}.
  \end{equation}
  \end{enumerate}
\end{lemma}
\begin{proof}
 Let us denote by $\mu$ the initial distribution of $Z_0$; hence,
 $P(\pi(Z_0)=i)=\mu(A_i)$.

\noindent $(1) \Longrightarrow (2)$. From the hypotheses
 \[
  q_{ij}=P(\pi(Z_1)=j|\pi(Z_0)=i)=\frac{\sum_{x \in A_i}P(\pi(Z_1)=j|Z_0=x) \mu(x)}{P(\pi(Z_0)=i)}
 \]
and this holds for every distribution $\mu$ (s.t.~$\mu(A_i)>0$) if and only if $q_{ij}=P(\pi(Z_1)=j|Z_0=x)$ for every $x \in A_i$,
that is, $\sum_{y \in A_j} p_{xy} = q_{ij}$ for all $x \in A_i$.

\noindent $(2) \Longrightarrow (1)$. If we compute $P(\pi(Z_n)=i_n| \pi(Z_{n-1})=i_{n-1}, \ldots, \pi(Z_0)=i_0)$ by means of
the decomposition
\[
 \{\pi(Z_n)=i_n, \pi(Z_{n-1})=i_{n-1}, \ldots, \pi(Z_0)=i_0\}=\bigcup_{z \in A_{i_0} \times A_{i_1} \times \cdots A_{i_n}}
 \{{Z_n =z_{n+1}, Z_{n-1} =z_{n} \cdots, Z_0 =z_1}\}
\]
and by using the Markov property of $Z_1,Z_2,\ldots$ and
equation~\eqref{eq:constant}
\[
 P \big (\pi(Z_n)=i_n| \pi(Z_{n-1})=i_{n-1}, \ldots, \pi(Z_0)=i_0 \big )=
 P \big (\pi(Z_n)=i_n| \pi(Z_{n-1})=i_{n-1} \big )=q_{i_{n-1} i_n}
\]
follows easily.
\end{proof}
From this result we can easily derive the most general transition
matrix of a 2-dimensional random walk $Z_n=(X_n,Y_n)$  with  state
space $\{(1,1),(1,0),(0,1),(0,0)\}$ whose marginals are Markov
chains. More precisely, given the marginals of $X$ and $Y$ with
state space $\A=\{0,1\}$
\[
\begin{pmatrix}
  p & 1-p \\
      q & 1-q \\
\end{pmatrix}
\quad
\begin{pmatrix}
 p^\prime & 1-p^\prime \\
 q^\prime & 1-q^\prime \\
 \end{pmatrix}
\]
the most general joint transition matrix can be easily obtained by applying Lemma~\ref{lem:partition} twice:
first with $A_1:=\{(1,1),(1,0)\}$, $A_2:=\{(0,1),(0,0)\}$
(to ensure that $X_n$ is a Markov chain) and then with $A_1:=\{(1,1),(0,1)\}$, $A_2:=\{(1,0),(0,0)\}$
(to ensure that $Y_n$ is a Markov chain).
The final result is
\begin{equation*}
\begin{pmatrix}
      p\lambda_1 & p(1-\lambda_1) & p^\prime-p\lambda_1 & 1+p\lambda_1-p^\prime-p \\
      p\lambda_2 & p(1-\lambda_2) & q^\prime-p\lambda_2 & 1+p\lambda_2-q^\prime-p \\
      q\mu_1 & q(1-\mu_1) & p^\prime-q\mu_1 & 1+q\mu_1-p^\prime-q \\
      q\mu_2 & q(1-\mu_2) & q^\prime-q\mu_2 & 1+q\mu_2-q^\prime-q \\
    \end{pmatrix}
\end{equation*}
with the constraints
\begin{align*}
&\lambda_1\in \Big[{p^\prime+p-1\over p}\vee 0,{p^\prime\over p}\wedge 1\Big],\quad
\lambda_2 \in \Big[{q^\prime+p-1\over p}\vee 0,{q^\prime\over p}\wedge 1\Big], \\
&\mu_1\in \Big[{p^\prime+q-1\over q}\vee 0,{p^\prime\over q}\wedge 1\Big],    \quad \mu_2
\in \Big[{q^\prime+q-1\over q}\vee 0,{q^\prime\over q}\wedge 1\Big].\end{align*}
This $4$-parameter model is sufficiently flexible and general to
cover a large variety of cases.  When $p=p^\prime$ and $q=q^\prime$,
i.e. $X$ and $Y$ have the same distribution, then the transition
matrix is simply
\begin{equation*}
\begin{pmatrix}
      p\lambda_1 & p(1-\lambda_1) & p(1-\lambda_1) & 1+p(\lambda_1-2) \\
      p\lambda_2 & p(1-\lambda_2) & q-p\lambda_2 & 1+p\lambda_2-q-p \\
      q\mu_1 & q(1-\mu_1) & p-q\mu_1 & 1+q\mu_1-p-q \\
      q\mu_2 & q(1-\mu_2) & q(1-\mu_2) & 1+q(\mu_2-2)\\
    \end{pmatrix}.
\end{equation*}
This is the case we are considering in the present subsection. In
what follows, without loss of generality, we shall assume that
$p\geq q$. The parameters $\lambda_i$ and $\mu_i$ regulate the
dependence between marginal sequences $X$ and $Y$. Clearly $X$ and
$Y$ are independent if and only if $\lambda_1=\mu_1=p$ and
$\lambda_2=\mu_2=q$. The transition matrix corresponding to that
particular choice of parameters will be denoted by $\mathbb{P}_{\rm
ind}$. If $\lambda_i$ and $\mu_i$ are maximal i.e.
$$ \lambda_1=1,\quad \lambda_2=q/p,\quad \mu_1=\mu_2=1,$$
then $X$ and $Y$ are (in a sense) maximally positive-dependent and
the corresponding transition matrix is
\begin{equation*}
\begin{pmatrix}      p & 0 & 0 & 1-p \\
      q & p-q & 0 & 1-p \\
      q & 0 & p-q & 1-p \\
      q & 0 & 0 & 1-q\\
    \end{pmatrix}.
\end{equation*}
We shall call this case {\it maximal dependence}, and the "maximal"
here means the maximal number of similar pairs $(1,1)$ or $(0,0)$.
The matrix above is not irreducible and therefore in the simulations
below, we shall use the following "nearly" maximal dependence matrix
\begin{equation}\label{max}
\mathbb{P}_{\rm max}(p,q)=\begin{pmatrix}
p-\epsilon & \epsilon & \epsilon & 1-p-\epsilon \\
      q & p-q & 0 & 1-p \\
      q & 0 & p-q & 1-p \\
      q-\epsilon & \epsilon & \epsilon & 1-q-\epsilon\\
    \end{pmatrix},\end{equation}
where to the end of the subsection $\epsilon=0.05$. Clearly the
distribution of $X$ and $Y$ is not affected by adding $\epsilon$.
The (nearly) maximal dependent $Z$ favors  pairs $(0,0)$ and
$(1,1)$. When $p$ and $q$ are relatively high, then typical outcome
of $Z$ will have many pairs $(1,1)$ and then changing a non
(1,1)-pair into a (1,1)-pair  has a tendency to increase the score.
\\
We shall also consider the "minimal dependence"   matrix that
corresponds to the small $\lambda_i$ and $\mu_i$. Such model favors
dissimilar pairs $(0,1)$ and $(1,0)$. Due to the fact that $X$ and
$Y$ sequences have the same transition matrix, unlike in the case of
maximal dependence, the minimal dependence (corresponding to the
correlation -1) is not always totally achieved and the structure of
minimal dependence matrix depends more on $p$ and $q$. In the
simulations we shall use the following minimal dependence matrices
(again $\epsilon$ is added to have an irreducible chain):
\begin{equation}\label{min}
\mathbb{P}_{\rm min}(p,q)=\begin{cases}
\begin{pmatrix}
2p-1 +\epsilon & 1-p-\epsilon & 1-p- \epsilon & \epsilon \\
      p+q-1 & 1-q & 1-p & 0 \\
      p+q-1 & 1-p & 1-q & 0 \\
      2q-1+\epsilon & 1-q-\epsilon & 1-q-\epsilon & \epsilon\\
    \end{pmatrix}, & \mbox{if $p+q > 1$ and $q \geq \frac{1}{2}$;}\\
\begin{pmatrix}
2p-1 +\epsilon & 1-p-\epsilon & 1-p- \epsilon & \epsilon \\
      p+q-1 & 1-q & 1-p & 0 \\
      p+q-1 & 1-p & 1-q & 0 \\
      \epsilon & q-\epsilon & q-\epsilon & 1-2q+\epsilon\\
    \end{pmatrix}, & \mbox{if $p+q > 1$ and $q < \frac{1}{2}$.}\\
    \end{cases}
\end{equation}
\newparagraph{The simulations.} Let us briefly describe the
simulations for a fixed transition matrix $\mathbb{P}$. First, let
us fix ${\bf A}\in \{0,1\}$. Then we fix a transition matrix
$\mathbb{P}$ and generate a Markov sequence $Z_1,\ldots,Z_{3 \cdot
7500}$ according to the stationary distribution corresponding to
$\mathbb{P}$. Denote
\begin{align*}
J_m:=\{j \colon j\leq m, Z_{3j-2}=Z_{3j}={\bf A}, Z_{3j-1} \neq {\bf A}\}.
\end{align*}
For each $m=100,200,...,7500$ we do the following procedure. If
$J_m= \emptyset$ we don't do anything and just pick the next $m$.
Suppose now that $J_m \neq \emptyset$ (obviously then also
$J_{m+100}\neq \emptyset$). We compute $l(m)$, the length of LCS of
$(Z_1,\ldots,Z_{3m})$. Next, for each $j=1,...,|J_m|$ we do the
following subprocedure. We compute $l(m,j)$, the length of LCS of
the sequence $$(Z_1,...,Z_{3j-2},\mathbf{A},Z_{3j},...,Z_{3m}).$$
Next we compute the difference $$r(m,j):=l(m,j)-l(m).$$ Note that
$r(m,j)\in \{-2,-1,0,1,2\}$. By the end of this subprocedure we have
$|J_m|$ values $r(m,1),\ldots,r(m,|J_m|)$ and we compute
$$E(m):={1\over |J_m|}\sum_{i=1}^{|J_m|}r(m,i).$$
Recall that $E_n=E[L(\mathcal{R}(Z_1,...,Z_n))|Z]-L(Z_1,...,Z_n)$.
Note that $E_{3m} \stackrel{d}{=} E(m)$, where $\mathcal{R}$ is the
random transformation used in the proof of \textbf{A3}, with
$\mathbf{B}=\mathbf{D}=\mathbf{A}$. The final goal is to to see
whether there are indications of the existence of a positive
$\epsilon_o$ such that $|E(m)|\geq \epsilon_o$ eventually.\\
\begin{figure}[!ht]
  \centering
    \includegraphics[width=0.5\textwidth,trim={0 0 1cm 0},clip]{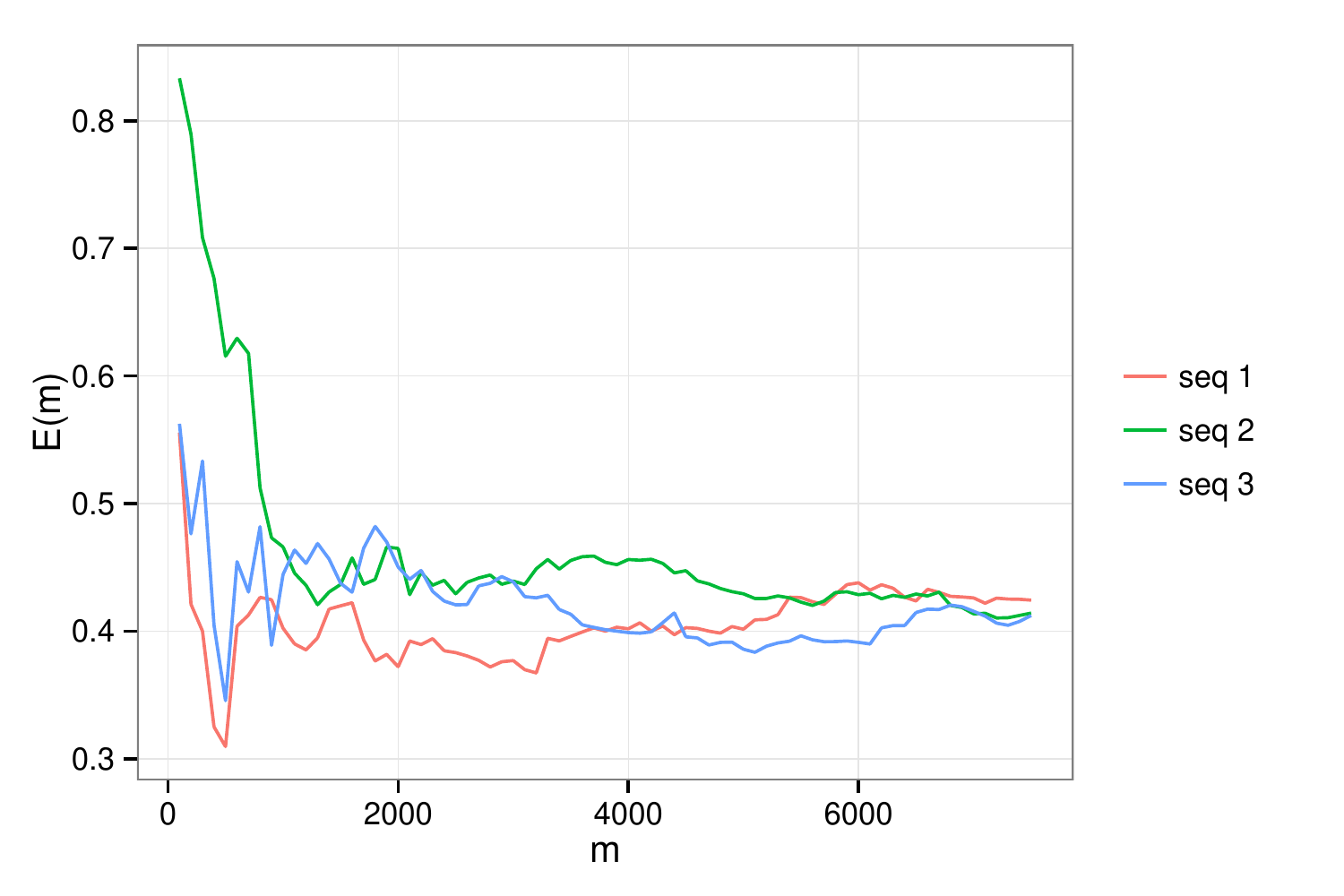}
      \caption{The behaviour of $E(m)$ with $\mathbb{P}=
      \mathbb{P}_{\mbox{\footnotesize{max}}}$, $p=0.9$, $q=0.7$, and $\mathbf{A}=(1,1)$. Note that for all three chains,
      $E(m)$ seems to converge to same positive and constant limit.}
\end{figure}\\
We start our simulations with ${\bf A}=(1,1)$. In Figure 1, three
different sequences  $Z_1,\ldots Z_{3\cdot 7500}$ are generated with
the same distribution corresponding to matrix $\mathbb{P}_{\rm
max}(0.9,0.7)$. In this case, for every $t$, $P(Z_t=(1,1))=0.819$
and turning a non-(1,1)-pair into a (1,1)-pair clearly has positive
effect to the score. From Figure 1, it is evident that $E(m)$ not
only is  bounded away from zero, but also converges to a strictly
positive constant limit (which we estimate to be around 0.4). The
convergence is not needed for {\bf A1} to hold, but based on that
picture, we conjecture that (at least for some models) $E_n$ a.s.
tends to a
constant limit. \\
Figure 1 also indicates that our choice of  $m$ is big enough to in
the sense that all different sequences behave similarly. Thus, in
what follows, we shall generate only one sequence for every
$\mathbb{P}$. In Figure 2, for several choices of $(p,q)$ three
different models, independent, maximal dependent (\ref{max}) and
minimal dependent (\ref{min}), are considered.
\begin{figure}
    \centering
    \begin{subfigure}[b]{0.5\textwidth}
        \includegraphics[width=\textwidth,trim={0 0 1cm 0},clip]{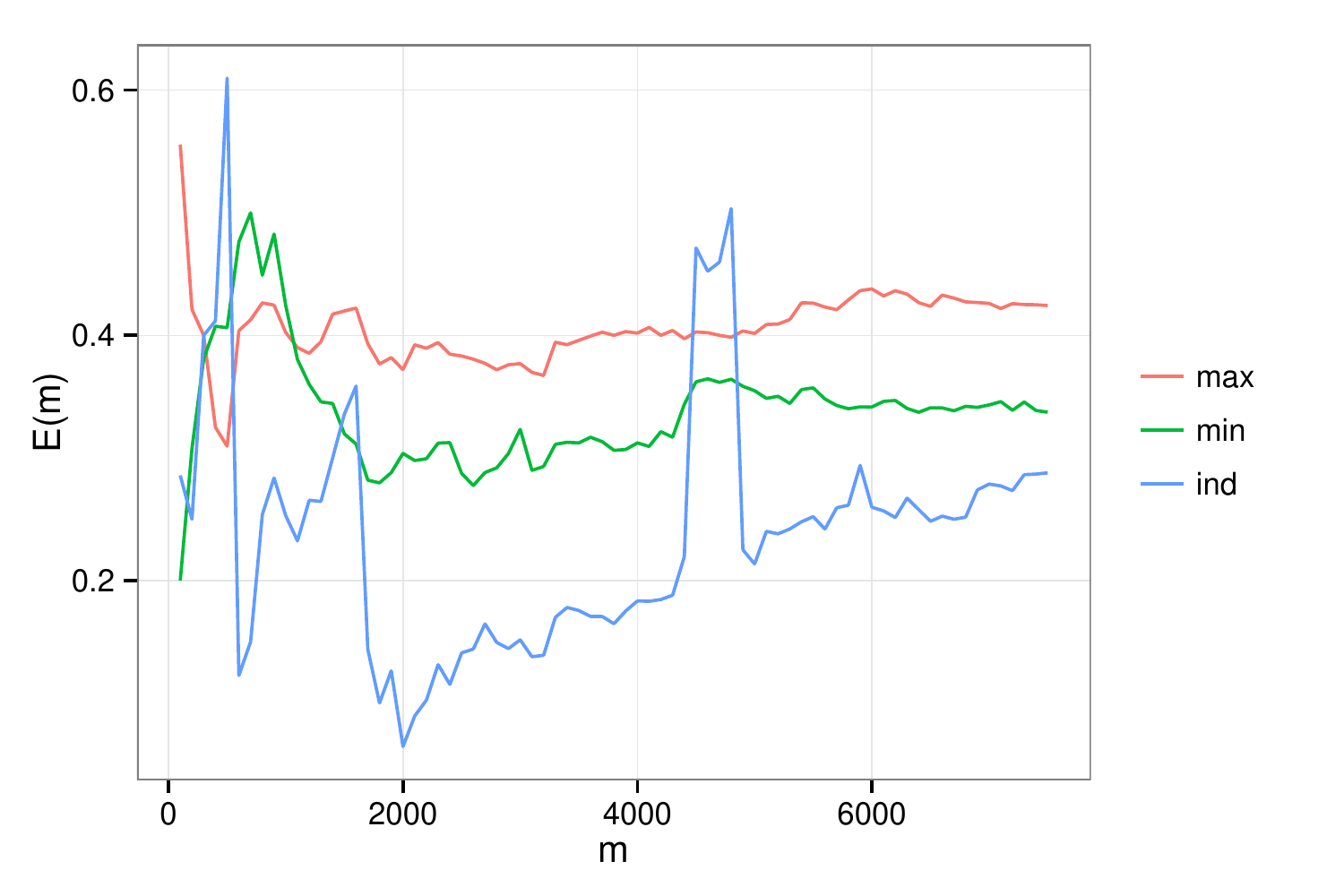}
        \caption{$p=0.9$, $q=0.7$}
    \end{subfigure}~
    \begin{subfigure}[b]{0.5\textwidth}
        \includegraphics[width=\textwidth,trim={0 0 1cm 0},clip]{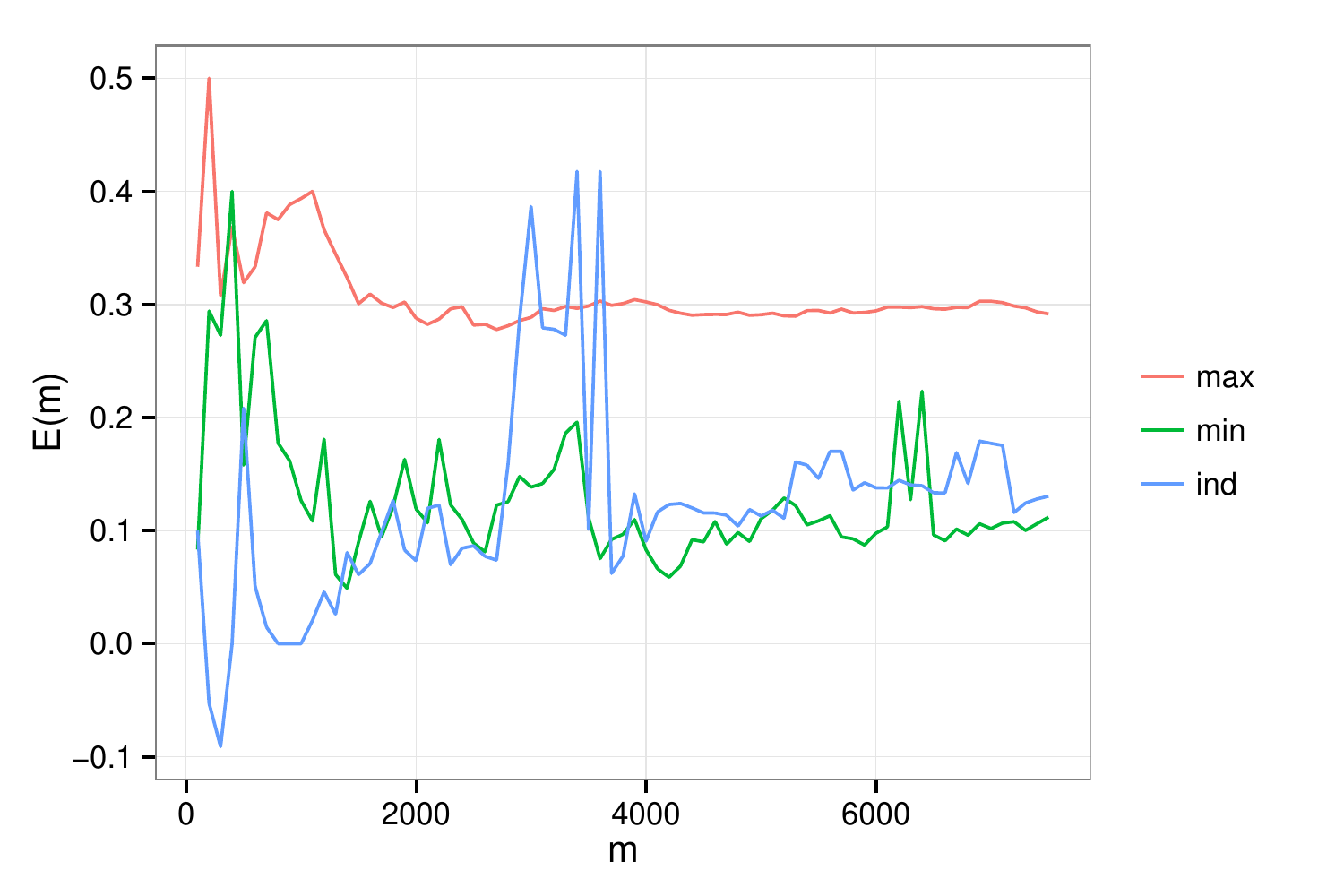}
        \caption{$p=0.8$, $q=0.6$}
    \end{subfigure}
 \begin{subfigure}[b]{0.5\textwidth}
        \includegraphics[width=\textwidth,trim={0 0 1cm 0},clip]{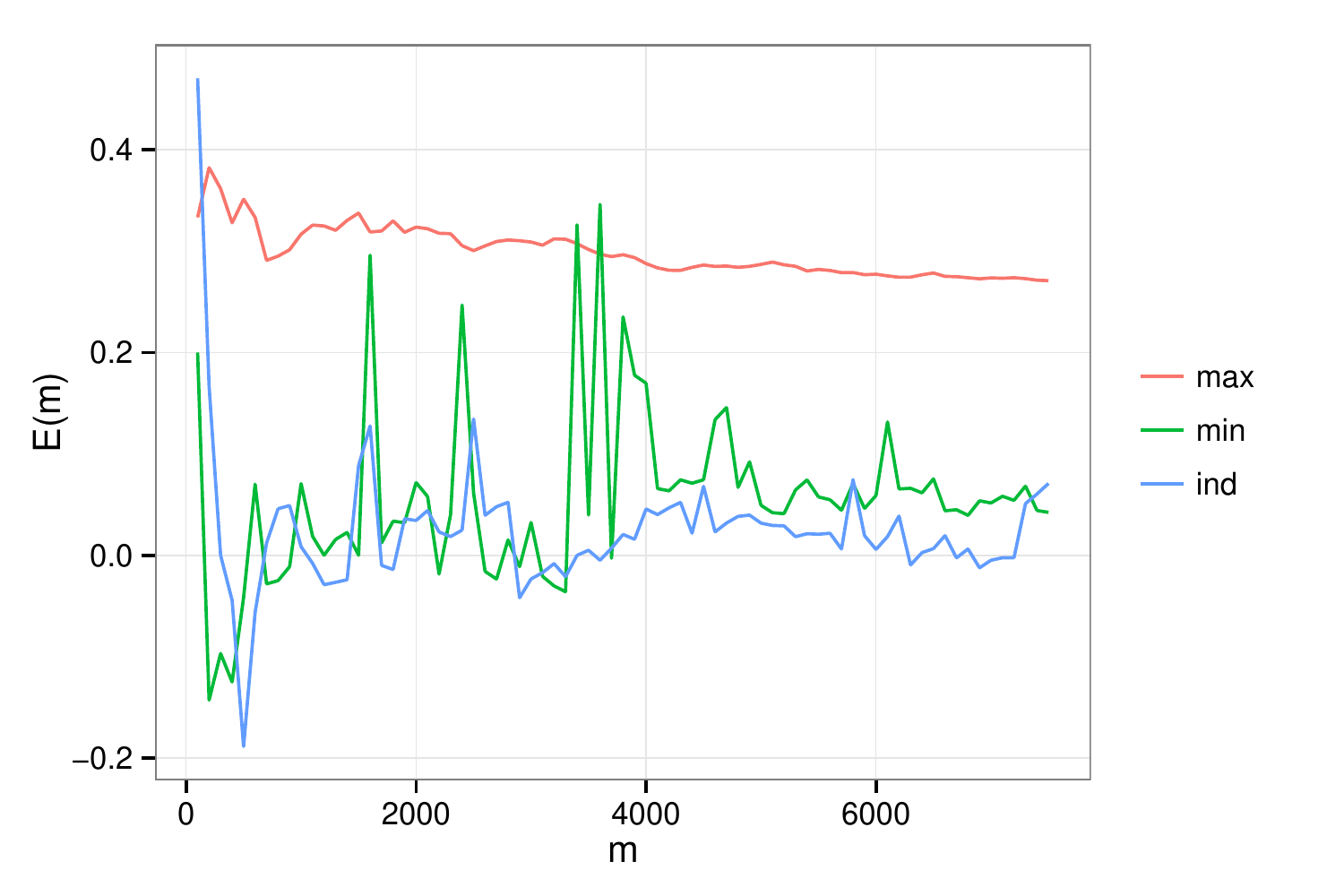}
        \caption{$p=0.7$, $q=0.7$} \label{fig:24}
    \end{subfigure}~
    \begin{subfigure}[b]{0.5\textwidth}
        \includegraphics[width=\textwidth,trim={0 0 1cm 0},clip]{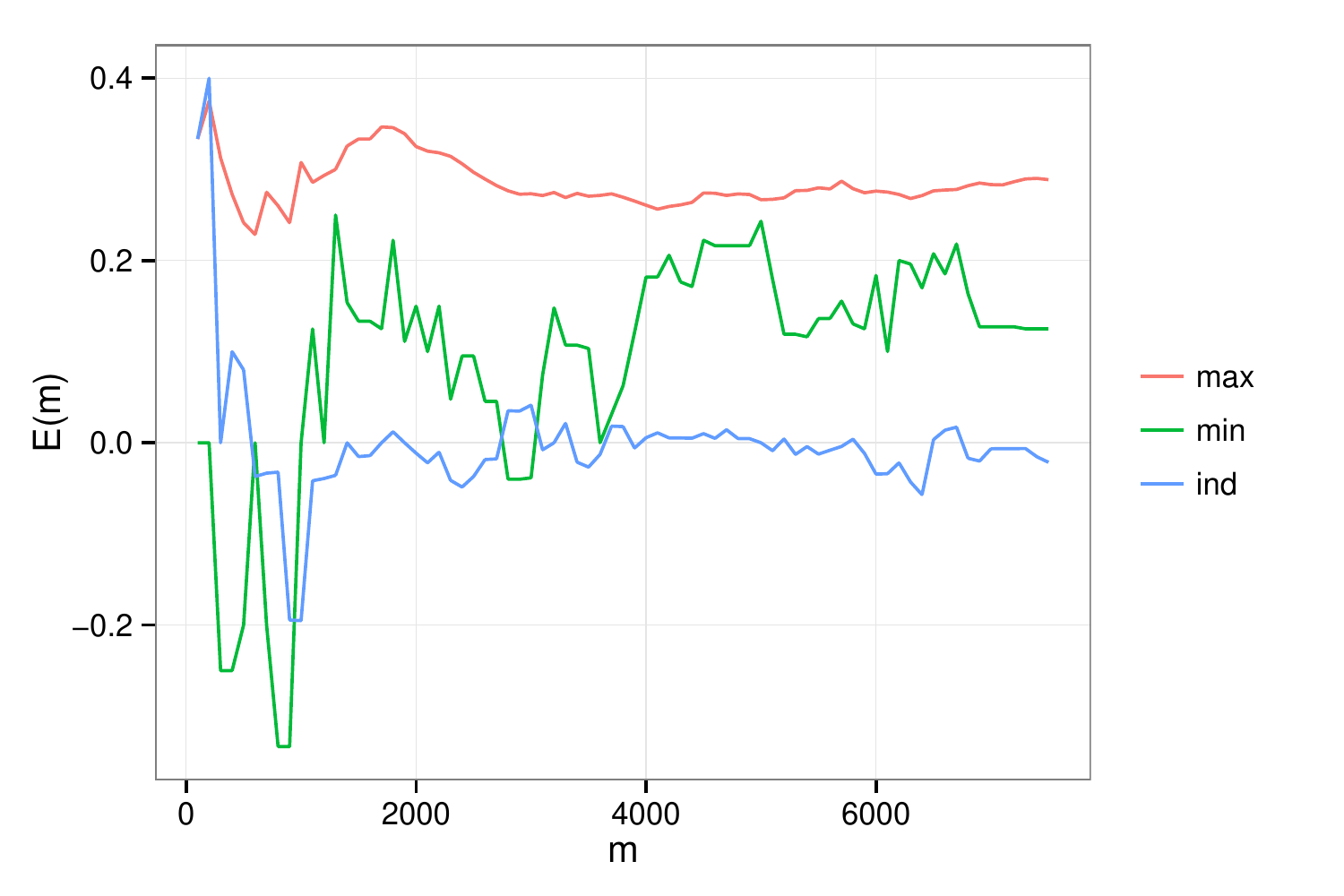}
        \caption{$p=0.7$, $q=0.4$}
    \end{subfigure}
    \caption{The behaviour of $E(m)$ with transition matrices
    $\mathbb{P}_{\mbox{\footnotesize{max}}}$, $\mathbb{P}_{\mbox{\footnotesize{min}}}$ and $\mathbb{P}_{\mbox{\footnotesize{ind}}}$, with $\mathbf{A}=(1,1)$.}
\end{figure}
From Figure 2, we see that in the case of $\mathbb{P}_{\rm max}$
(where the probability of $(1,1)$-pair is the highest, corresponding to the red line)
$E(m)$ clearly is bounded away from zero for every $(p,q)$. For
independent sequences and the sequences corresponding to
$\mathbb{P}_{\rm max}$, the desired boundedness is evident for
models with relatively big $p$ and $q$ (upper row), whilst for
smaller $p$ and $q$, it  might not be so. This  is due to relatively
low number of (1,1)-pairs. Indeed, for independent sequences the
probability of (1,1)-pair is $pq$ and so if $p=0.7$ and $q=0.4$ (D),
the proportion of (1,1)-pairs is  too small for our random
transformation to have positive effect to the score.
\begin{figure}
    \centering
    \begin{subfigure}[b]{0.5\textwidth}
        \includegraphics[width=\textwidth,trim={0 0 1cm 0},clip]{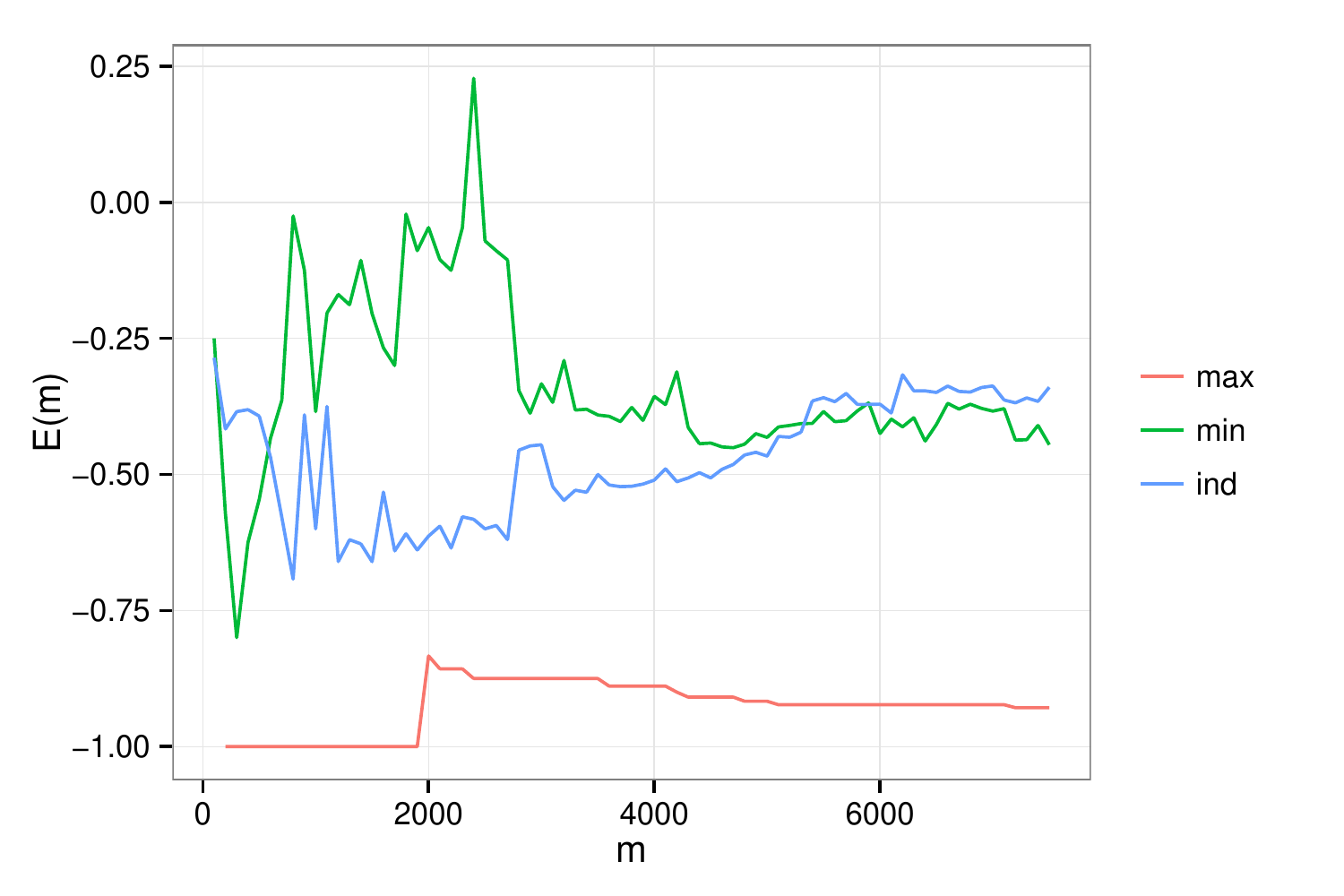}
        \caption{$p=0.7$, $q=0.7$}
    \end{subfigure}~
    \begin{subfigure}[b]{0.5\textwidth}
        \includegraphics[width=\textwidth,trim={0 0 1cm 0},clip]{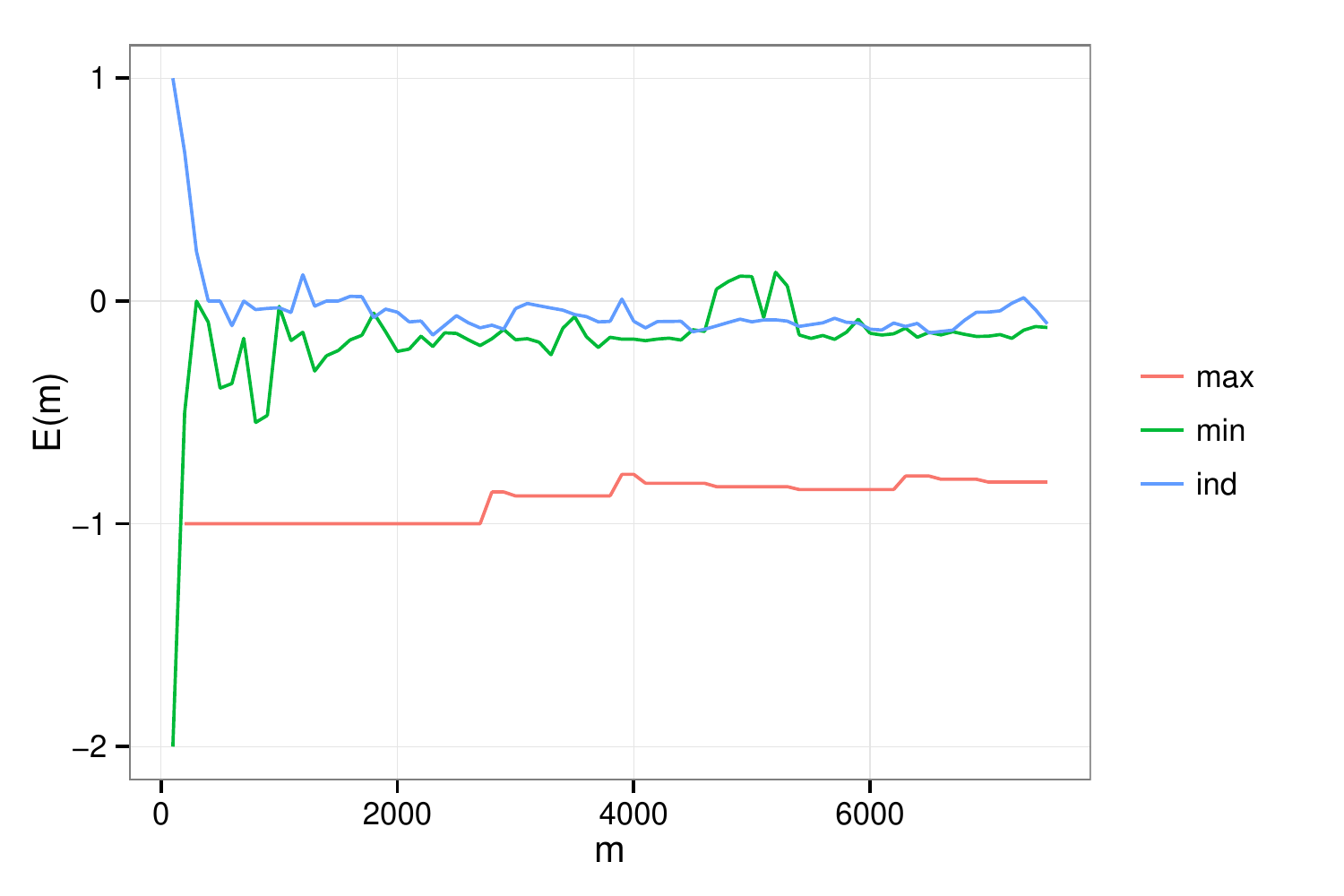}
        \caption{$p=0.7$, $q=0.4$}
    \end{subfigure}
    \caption{The behavior of $E(m)$ with transition matrices $\mathbb{P}_{\mbox{\footnotesize{max}}}$, $\mathbb{P}_{\mbox{\footnotesize{min}}}$ and $\mathbb{P}_{\mbox{\footnotesize{ind}}}$, with $\mathbf{A}=(0,1)$.}
\end{figure}\\
The random transformation considered so far is designed to increase
the score and for most of the models in Figure 2, it indeed does so.
We next consider a new ${\mathcal R}$ that tends to decrease the score.
For that, we just take ${\bf A}=(0,1)$.  In Figure 3, we repeat, with this new ${\mathcal
R}$, the
same simulations of cases (C) and (D) of Figure 2.
The choice of these cases is due to the fact that, for $\mathbb{P}_{\rm min}$ and
$\mathbb{P}_{\rm ind}$, the former transformation ${\mathcal R}$ (with ${\bf
A}=(1,1)$) did not convincingly show the existence of the positive lower bound
$\epsilon_0$. For the
independent marginals case $(p=q=0.7)$, the behavior of $E(m)$ is
much better now and we can conclude that  $E(m)$ converges a.s.
to a constant limit that for cases $\mathbb{P}_{\rm min}$ and
$\mathbb{P}_{\rm ind}$ are in $(-0.25,-0.5)$. Recall that the
negative limit also ensures {\bf A1}, we just formally have to
consider a different score function. In the other case, namely the case (B) of
Figure 3, we see indications of the convergence of $E(m)$, but for $\mathbb{P}_{\rm
min}$ and $\mathbb{P}_{\rm ind}$, it is difficult to conclude whether
the limit is different from zero or not.\\\\
In the case (A) of Figure 3, the probability $P(Z_t=(0,1))$ is 0.045
(max), 0.28 (min) and 0.49 (ind). The same probabilities in the case (B)
are 0.063, 0.418 and 0.245 respectively. We see that, especially
for $\mathbb{P}_{\rm max}$, the number of $(0,1)$-pairs in  the
sequence is very small and that jeopardies the simulations in this
case. The small number of $(0,1)$ pairs is evident from the
pictures, where the red line is not varying much. Therefore, we
combine the transformations by taking
$\mathbf{A}_1=\mathbf{B}_1=\mathbf{D}_1=(1,0)$ and
$\mathbf{A}_2=\mathbf{B}_2=\mathbf{D}_2=(0,1)$. As before, we
generate a Markov sequence $Z_1,\ldots,Z_{3 \cdot 7500}$ according
to the stationary distribution. We then apply the procedure
described above twice: first with $\mathbf{A}=(1,0)$, and then with
$\mathbf{A}=(0,1)$. In this way we obtain the sets
 $J_m^1$ and the LCS-differences $r_1(m,i)$ (corresponding to the pair $(1,0)$), and the sets $J_m^2$ and the
 LCS-differences $r_2(m,i)$ (corresponding to the pair $(0,1)$). Finally we define
 $$E(m):={1\over |J_m^1|+|J_m^2|} \left(\sum_{i=1}^{|J_m^1|}r_1(m,i) + \sum_{i=1}^{|J_m^2|}r_2(m,i) \right).$$
 Again, note that $R_{3m} \stackrel{d}{=} E(m)$, where $\mathcal{R}$
 is now the combined random transformation with $\mathbf{A}_1=\mathbf{B}_1=\mathbf{D}_1=(1,0)$
 and $\mathbf{A}_2=\mathbf{B}_2=\mathbf{D}_2=(0,1)$. When we described the combined transformation in Section~\ref{subsec:combined},
 we mainly considered the case 
 $q_1=q_2$: this is true for our transition matrices $\mathbb{P}_{\mbox{\footnotesize{max}}}$,
  $\mathbb{P}_{\mbox{\footnotesize{min}}}$,  $\mathbb{P}_{\mbox{\footnotesize{ind}}}$, so the use of combined random transformations in the simulations is justified
  \footnote{More specifically, note that when $|{\mathbb{A}}|=2$, then, as it is easy to see, the following conditions are sufficient for
  $q_1=q_2$ to hold: $\mathbb{P}_{22}=\mathbb{P}_{33}$, $\mathbb{P}_{23}=\mathbb{P}_{32}$, $\mathbb{P}_{21}=\mathbb{P}_{31}$, $\mathbb{P}_{12}=
  \mathbb{P}_{13}$, $\mathbb{P}_{24}=\mathbb{P}_{34}$, $\mathbb{P}_{42}=\mathbb{P}_{43}$. The transition matrices $\mathbb{P}_{\mbox{\footnotesize{max}}}$,
  $\mathbb{P}_{\mbox{\footnotesize{min}}}$, $\mathbb{P}_{\mbox{\footnotesize{ind}}}$ satisfy those equalities.}.
  The results of these new simulations 
  are presented in Figure
  4.  Since in all cases $P(Z_t=(0,1))=P(Z_t=(1,0))$, including (1,0) into $\mathcal{R}$ has the same effect as doubling the number of simulations in Figure 3.   We see that the red line now
  varies more and we can believe that there is a convergence. In the case $p=q=0.7$, the convergence of
  green and blue lines to the limits around -0.4 is now even more
  evident, and for the most difficult case $p=0.7, q=0.4$, we now can
  deduce that $\lim\sup_m E(m)<0$, i.e. {\bf A1} also holds in this
  case.
\\
  \begin{figure}
    \centering
 \begin{subfigure}[b]{0.5\textwidth}
        \includegraphics[width=\textwidth,trim={0 0 1cm 0},clip]{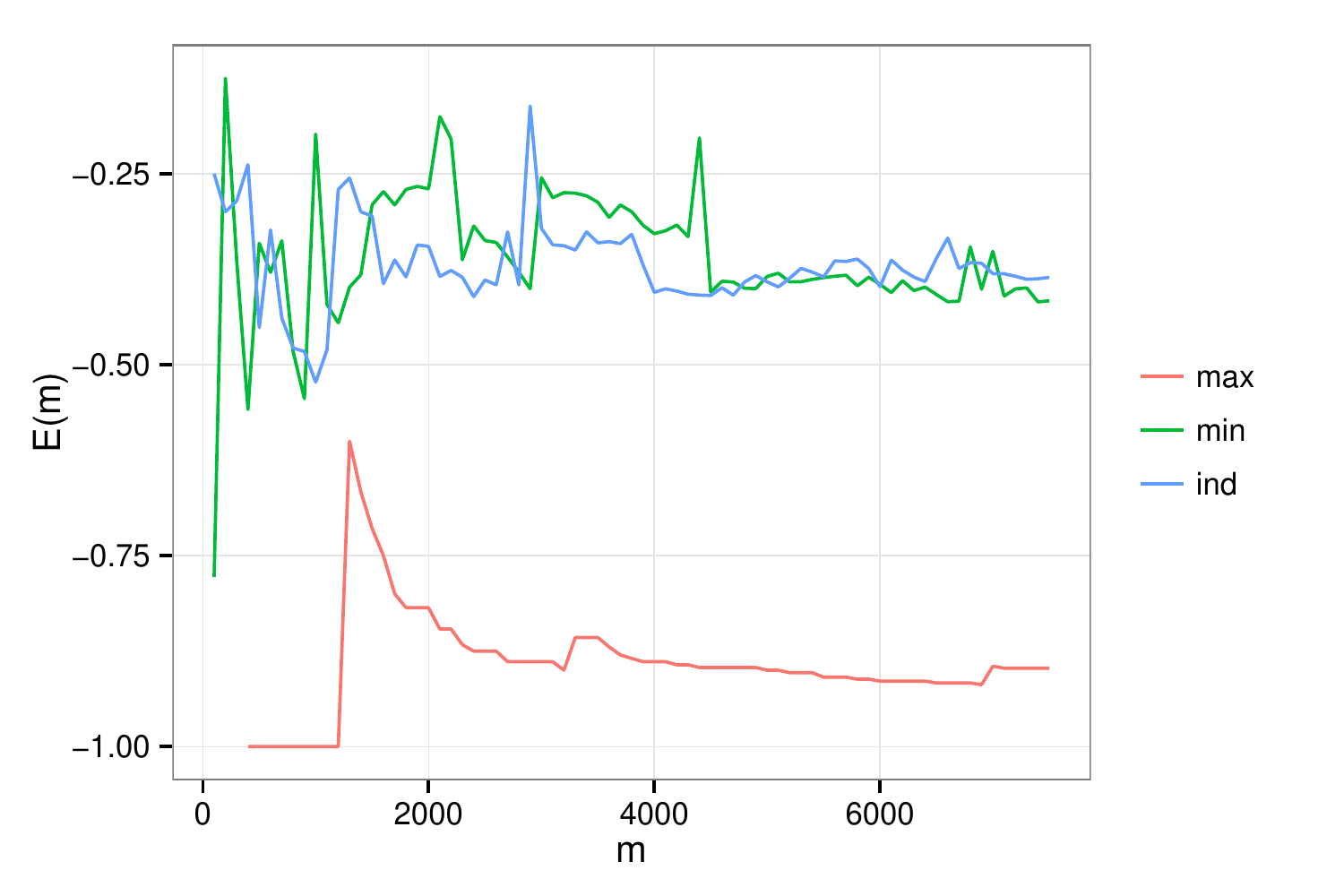}
        \caption{$p=0.7$, $q=0.7$}
    \end{subfigure}~
    \begin{subfigure}[b]{0.5\textwidth}
        \includegraphics[width=\textwidth,trim={0 0 1cm 0},clip]{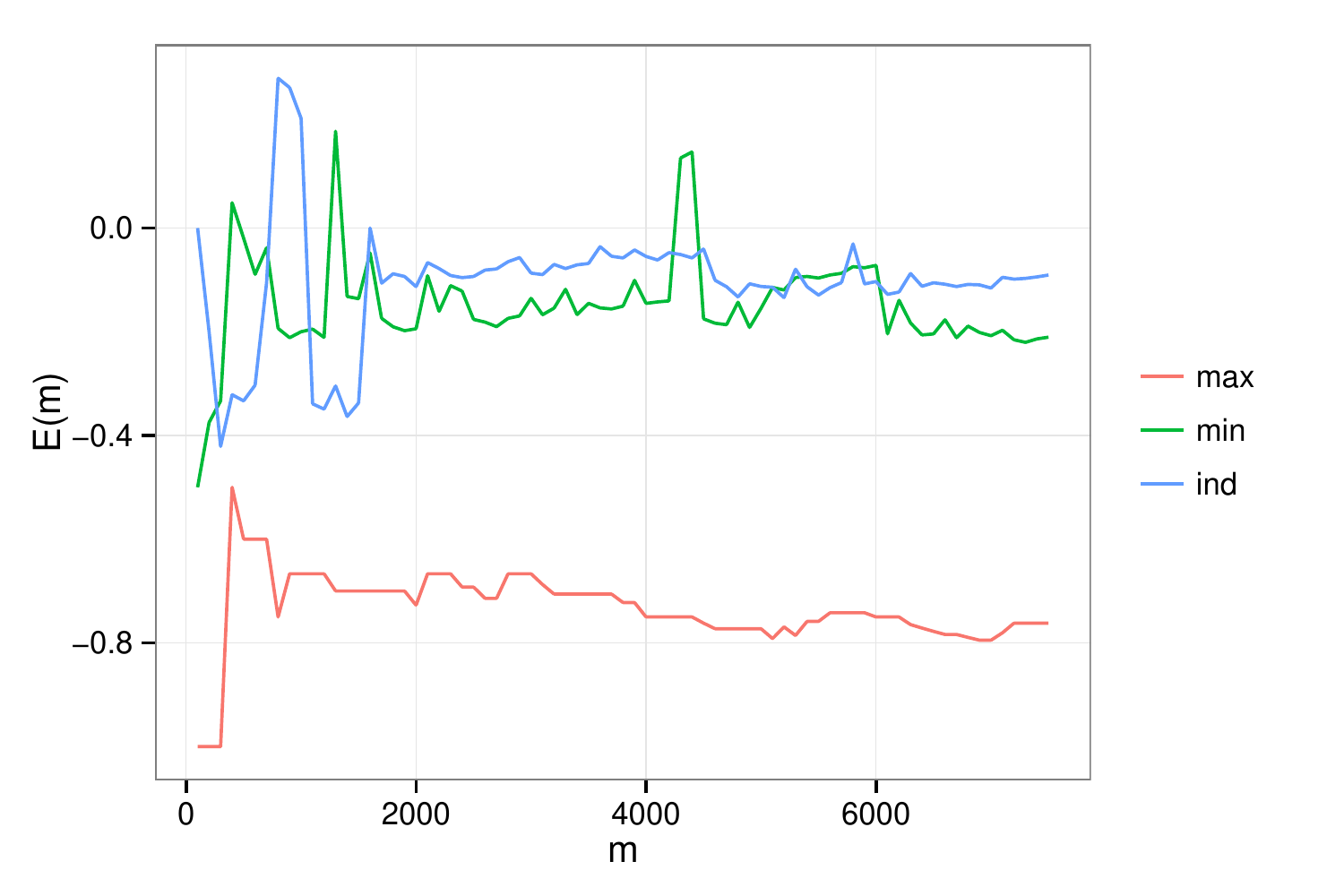}
        \caption{$p=0.7$, $q=0.4$}
    \end{subfigure}~
    \caption{The behaviour of $E(m)$ with transition matrices $\mathbb{P}_{\mbox{\footnotesize{max}}}$,
    $\mathbb{P}_{\mbox{\footnotesize{min}}}$ and $\mathbb{P}_{\mbox{\footnotesize{ind}}}$ using combined random transformations:  $\mathbf{A}_1=\mathbf{B}_1=\mathbf{D}_1=(1,0)$ and $\mathbf{A}_2=\mathbf{B}_2=\mathbf{D}_2=(0,1)$.}
\end{figure}

\section{The upper bound}\label{upper}
In order to judge the sharpness of the lower bound, we briefly
calculate the upper bounds of $\Phi\big(|L_n-EL_n|\big)$ in the case
$\phi(x)=x^r$. In the case of independent random variables, there
are many ways of finding upper bound starting from Efron-Stein type
of inequalities when $r=2$. For an overview of several methods for
obtaining the upper bound, see \cite{HoudreLember}. However, most of
the methods assume independence of random letters.   In the case of
PMC-model, probably the easiest way to get an upper bound of the
correct order seems to be via the following McDiarmid's-type of
inequality for Markov chains (see \cite[Corollary 2.9]{paulin})
\begin{theorem} Let $Z:=Z_1,\ldots,Z_n$ be a
homogeneous Markov chain with state space $\mathcal{Z}$ and mixing time
$t(\epsilon)$. Let $f: \mathcal{ Z}^n\to \mathbb{R}$ be function
satisfying the bounded difference inequality; for every $z, z'\in
\mathcal{ Z}^n$
$$|f(z)-f(z')|\leq \sum_{i=1}^n c_i I_{\{z_i\ne z'_i\}},$$
where $c:=(c_1,\ldots,c_n)$ are some nonnegative constants. Then,
for any $s>0$
\begin{equation}\label{McDiarmid}
P\Big(|f(Z)-Ef(Z)|>s\Big)\leq 2\exp\Big[-{ s^2\over 8\|c\|^2 t_{mix}}\Big],
\end{equation}
where $\|c\|^2=\sum_i c_i^2$ and  $t_{mix}=t(1/4)$.\end{theorem}
We are going to apply this theorem for $f=L$. Since the change of a
value of $Z_i$ changes the score by at most $2\Delta$, we have the
bounded difference property with $c_i=2\Delta$ and $\|c\|^2=n4\Delta^2$.
Since, by assumption $Z$ is aperiodic, there exists $m\geq 1$ such
that
$$\min_{\mathbf{A},\mathbf{B}\in \A\times \A}P(Z_{1+m}=\mathbf{B}|Z_1=\mathbf{A})=:p_o>0.$$
Then, as it is well-known,
$$\max_{\mathbf{A}\in \A\times \A}\|\pi(\cdot)-P(Z_{1+n}\in \cdot|Z_1=\mathbf{A})\|\leq C
\rho^n,$$ where $\pi$ is stationary distribution of $Z$, $\|\cdot\|$
is total variation distance, $\rho:=(1-|\A|^2p_o)^{1\over m}$ and
$C=1$ if $r=1$, and $C:=(1-|\A|^2p_o)^{-1}$, otherwise.
 Therefore
$$t(\epsilon)\leq {\ln{\epsilon}-\ln C\over \ln(\rho)}<\infty,\quad t_{mix}\leq {-(\ln{4}+\ln C)\over \ln(\rho)}<\infty.$$
Applying now equation~\eqref{McDiarmid}, we get
\begin{align}\label{eq:HoeffdingIneqOptScore}
{P}\left(\left|L(Z)-{E}\left(L(Z)\right)\right|\geq s\right)\leq 2\exp\left(-\frac{s^{2}}{nF}\right),
\end{align}
where $F:=32\Delta^2t_{mix}.$ From that it is trivial to get the upper bound.
Take $W_n=|L(Z)-{E}\left(L(Z)\right)|$
\begin{align*}
{E}\left(W_{n}^{r}\right)=\int_{0}^{\infty}{P}\left(W_{n}\geq t^{\frac{1}{r}}\right)dt
\le x+2\int_{x}^{\infty}\exp\left(-\frac{t^{\frac{2}{r}}}{nF}\right)dt.
\end{align*}
Minimizing in $x$, i.e., taking $x=\left(F(\ln 2)n\right)^{{r/2}}$,
and changing variables $u=t^{2/r}/(Fn)$, lead to:
\begin{align*}
{E}\left(W_{n}^{r}\right)\leq\left(F(\ln 2)n\right)^{\frac{r}{2}}+rF^{r/2}n^{\frac{r}{2}}\int_{\ln 2}^{\infty}e^{-u}u^{\frac{r}{2}-1}du
=n^{\frac{r}{2}}F^{r/2}\left[(\ln 2)^{\frac{r}{2}}+r\int_{\ln 2}^{\infty}e^{-u}u^{\frac{r}{2}-1}du\right],
\end{align*}
an upper bound of the form $C(r)\,n^{r/2}$, where
\begin{align*}
C(r):=F^{r/2}\left[(\ln 2)^{\frac{r}{2}}+r\int_{\ln 2}^{\infty}e^{-u}u^{\frac{r}{2}-1}du\right].
\end{align*}
When $x=0$, the corresponding constant is slightly bigger than
$C(r)$, and is given by:
\begin{align*}
D(r):=rF^{r/2}\int_{0}^{\infty}e^{-u}u^{\frac{r}{2}-1}du=rF^{r/2}\,\Gamma\left(\frac{r}{2}\right).
\end{align*}
\subsection{Appendix: proof or Theorem~\ref{general}}
Let $B_n \subset \Z_n$ be the set of outcomes of $Z$ such that
$$\big\{E[L(\tZ)-L(Z)|Z]\geq \epsilon_o\}=\{Z\in
B_n\}.$$ Let the set $\V^o\subset \S^V$ be defined as follows:
\begin{equation}
v\in \V^o \quad \Leftrightarrow \quad P(Z\not \in B_n|V=v)\leq
\sqrt{\Delta_n}.
\end{equation}
Now
$$
\Delta_n \geq P(Z\not \in B_n)\geq \sum_{v\not\in \V^o}P(Z\not \in
B_n|V=v)P(V=v)>\sqrt{\Delta_n}P(V\not \in \V^o),\quad \Rightarrow
\quad P(V\not \in \V^o)\leq \Delta_n^{1\over 2}.$$
Furthermore, for every $v\in \V^o$, let $\U^o(v)\subset \S(v)$ be
defined as follows
\begin{equation}
u\in \U^o(v) \quad \Leftrightarrow \quad P(Z\not \in
B_n|V=v,U=u)\leq {\Delta_n}^{1\over 4}.
\end{equation}
Again,
\begin{align*}
\sqrt{\Delta_n} & \geq P(Z\not \in B_n|V=v)\geq \sum_{u\not\in \U^o(v)}P(Z\not \in B_n|V=v, U=u)P(U=u|V=v)\\
& >{\Delta_n}^{1\over 4}P(U\not \in \U^o(v)|V=v),\quad \Rightarrow
\quad P(U\not \in \U^o(v)|V=v)\leq \Delta_n^{1\over 4}.\end{align*}
We now show that there exists $n_o$ so big that when $v\in \V^o\cap
\V$ and $u\in \U(v)\cap \U^o(v)$, then
\begin{equation}\label{luv}
l(u+1,v)-l(u,v)\geq {\e_o\over 2}.\end{equation}
 Let
$Z_{(u,v)}$ be a random vector having the distribution $P_{(u,v)}$.
By {\bf A3}, thus,
$$l(u+1,v)=E[L(\mathcal{R}(Z_{(u,v)}))].$$
Hence
\begin{align*}
l(u+1,v)-l(u,v)&=E[L(\mathcal{R}(Z_{(u,v)}))]-E[L({Z}_{(u,v)})]=E[L(\mathcal{R}(Z_{(u,v)}))-L({Z}_{(u,v)})]\\
&=E\big(E[L(\mathcal{R}(Z_{(u,v)}))-L({Z}_{(u,v)})\big|Z_{(u,v)}]\big).\end{align*}
By assumption {\bf A2}, for any pair of sequences $z$, the worst
decrease of the score, when applying the random transformation is $-
A$. Hence,
\begin{equation*}\label{auto}
E\big(E[L(\mathcal{R}(Z_{(u,v)}))-L({Z}_{(u,v)})\big|{Z}_{(u,v)}]\big)\geq
\e_o P\big({Z}_{(u,v)} \in B_n \big)-A P({Z}_{(u,v)}\not \in
B_n)\geq \e_o(1-\Delta_n^{1\over 4})-A\Delta_n^{1\over
4}.\end{equation*} The last inequality follows from the fact that by
definition of $\U^o(v)$, when
 $v\in \V^o$ and $u\in \U^o(v)$, it holds
$$P\big({Z}_{(u,v)} \in B_n\big)=P(Z\in B_n|V=v,U=u)\geq
1-{\Delta_n}^{1\over 4}.$$ Since $\Delta_n\to 0$, there exists $n_o$
so big that $\e_o(1-\Delta_n^{1\over 4})-A\Delta_n^{1\over 4}\geq
{\e_o\over 2}$, provided $n>n_0$. In what follows, we assume $n>n_o$.\\\\
Fix   $v\in \V^o\cap \V$ and consider the set $\U(v)\cap \U^o(v)$
and $n>n_o$. When $u\in \mathcal{ U}_n(v) \cap \U^o(v)$, then by
equation~\eqref{luv} $ l(u+1,v)-l(u,v)\geq\frac{\e_{o}}{2}.$ When
$u\not \in \mathcal{ U}_n(v) \cap \U^o(v)$, then $l(u+1,v)-l(u,v)\geq
-A$.
\\\\
Recall $\mathcal{ U}_n(v)=\{u_n(v)+1,\ldots,u_n(v)+m_n(v)\}$.   The set
$\mathcal{ U}_n(v) \cap \U^o(v)$  can be represented as the union of
disjoint intervals of  $\mathcal{ U}_n(v)$
 , i.e.
$$\mathcal{ U}_n(v) \cap \U^o(v)=\bigcup_{j=1}^{k(v)} I_j(v),$$
where $$I_j(v)=\{u_n(j,v)+1,\ldots,u_n(j,v)+m_n(j,v)\}$$ is a
subinterval of $\mathcal{ U}_n(v)$. Obviously the number of intervals
$k(v)$ as well as the intervals $I_j(v)$ depend on $n$. On every
interval $I_j(v)$, the function $l(\cdot,v)$ increases with the
slope at least $\frac{\e_{o}}{2}$ i.e.
\begin{equation}\label{intervals}
\text{ if }u\in I_j(v)\text{, then  }
l(u+1,v)-l(u,v)\geq\frac{\e_{o}}{2}.
\end{equation}
 Let us consider the sets
$$J_j(v):=\{l(u_n(j,v)+1,v),\ldots,l(u_n(j,v)+m_n(j,v),v)\}\quad j=1,\ldots,k(v).$$ Thus $J_j(v)$
is is the image of the set $I_j(v)$ when applying $l(\cdot,v)$. Note
that if $u=u_n(j,v)+m_n(j,v)$, i.e. $u$ is the last element in the
interval,  then $l(u+1,v)$ is outside of the interval $J_j(v)$. We
know that all elements of $J_j(v)$ are at least
$\frac{\e_{o}}{2}$-apart from each other. However, the
intervals $J_j(v)$ might overlap (even thought we know that the intervals
$I_j(v)$ do not). Since for any $u\in \mathcal{ U}_n(v) \backslash
\U^o(v)$, it holds that $l(u+1,v)-l(u,v)\geq -A$, we have
\begin{equation}\label{decrease}
\sum_{u\in \mathcal{ U}_n(v) \backslash
\U^o(v)}\big(l(u+1,v)-l(u,v)\big)\geq -A|\mathcal{ U}_n(v) \backslash
\U^o(v)|.
\end{equation}
The inequality equation~\eqref{decrease} together with equation~\eqref{intervals}
implies that the sum of the lengths of (integer) intervals $J_j(v)$
differs from the length of the set $J(v):=\bigcup_{j=1}^{k(v)} J_j(v)$
at most by $A|\mathcal{ U}_n(v) \backslash \U^o(v)|$. Formally, defining
for any finite set of real numbers $T$ the length $\ell (T)$ as the
difference between maximum and minimum element of $T$ i.e.
$$\ell (J_j(v)):=l(u_n(j,v)+m_n(j,v),v)-l(u_n(j,v)+1,v),$$  we obtain
\begin{equation}
\sum_{j=1}^{k(v)}\ell (J_j(v))-\ell (J(v))\leq \sum_{j=1}^{k(v)}\ell
(J_j(v))-\Big(l(u_n(k,v)+m_n(k,v),v)-l(u_n(1,v)+1,v)\Big)\leq
A|\mathcal{ U}_n(v) \backslash \U^o(v)|.
\end{equation}
The first inequality follows from the fact that
$$l(u_n(k,v)+m_n(k,v),v)-l(u_n(1,v)+1,v)\leq \ell (J(v))$$
and the second from equations~\eqref{decrease} and \eqref{intervals}.
\\\\
The number of $\frac{\e_{o}}{2}$-apart points
needed for covering an (real) interval with length  $A|\mathcal{ U}_n(v)
\backslash \U^o(v)|$ is at most $${2A|\mathcal{ U}_n(v) \backslash
\U^o(v)|\over \e_{o}}+1.$$ This means that due to the
overlapping at most ${2A|\mathcal{ U}_n(v) \backslash \U^o(v)|\over
\e_{o}}+1$  points that are
$\frac{\e_{o}}{2}$-apart  will be lost implying that in the
set $J(v)$ there are at least
$$|\mathcal{ U}_n(v)|-{2A|\mathcal{ U}_n(v) \backslash \U^o(v)|\over
\e_{o}}-1=m_n(v)-{2A|\mathcal{ U}_n(v) \backslash \U^o(v)|\over
\e_{o}}-1$$ points that are (at least)
$\frac{\e_{o}}{2}$-apart from each other.
\\\\
Using the inequality (recall $v\in \V^o$)
 $$P(U\not \in \U^o(v)|V=v)\leq \Delta_n^{1\over 4}$$
and equation~\eqref{varphi} we obtain
\begin{align*}
 {\Delta_n}^{1\over 4}\geq {P}(U\in \mathcal{ U}_n(v)\backslash \mathcal{
 U}^o_n(v)|V=v)=\sum_{u\in \mathcal{ U}_n(v)\backslash \mathcal{
 U}^o_n(v)} {P}\left(U=u\right|V=v)\geq |\mathcal{ U}_n(v)\backslash
 \U^o(v)|\varphi_v(n)
\end{align*}
implying that $$ |\mathcal{ U}_n(v)\backslash
 \U^o(v)|\leq {\Delta_n}^{1\over 4}\varphi_v(n)^{-1}. $$
Thus by {\bf A4}, there exists $n_1$ such that
$$m_n(v)-{2A|\mathcal{ U}_n(v) \backslash \U^o(v)|\over
\e_{o}}-1\geq m_n(v)-2A{{\Delta_n}^{1\over 4}\over
\e_o\varphi_v(n)}-1\geq
{(c(v)\epsilon_o-2A{\Delta_n}^{1\over
4}-\epsilon_o\varphi_v(n))\over \epsilon_o\varphi_v(n)}\geq
{r_v(n)\over \varphi_v(n)},\quad \forall n>n_1$$ where
$$r_v(n):=c(v)-{2A\Delta^{1\over 4}_n\over \epsilon_o}-\varphi_v(n)\to c(v)$$
uniformly with respect to $v \in \V$
(for the definition of uniform convergence with respect to a variable in a sequence of sets,
 see for instance \cite[Definition 2.2]{cf:BZ2003}).
To summarize: the set  $$J(v)\subseteq
\{l(u_n(v)+1,v),\ldots,l(u_n(v)+m_n(v),v)\}$$ contains at least
${r_v(n)\over \varphi_v(n)}$ elements being
$\frac{\e_{o}}{2}$-apart from each other.
\\\\
Finally, define the set
$$\mathcal{ A}_n(v):=\Big\{ u \in \mathcal{ U}_n(v):\quad  |l(u,v)-\mu_n|\geq {\e_o
r_v(n)\over 8 \varphi_v(n)} \Big \}.$$ Since the interval
$$\Big[\mu_n- {\e_o
r_v(n)\over 8 \varphi_v(n)},\mu_n+ {\e_o r_v(n)\over 8
\varphi_v(n)}\Big]$$ contains at most ${r_v(n) \over 2
\varphi_v(n)}+1$ elements that are $\frac{\e_{o}}{2}$-apart
from each other and $J(v)$ contains at least ${r_v(n)\over
\varphi_v(n)}$ of such elements, it follows that that the set
$$\mathcal{ B}_n(v):=\{l(u,v): u\in \mathcal{ A}_n(v)\}$$ contains at least  ${r_v(n)\over 2\varphi_v(n)}-1$
points being $\frac{\e_{o}}{2}$-apart from each other, and
in particular, the set  $\mathcal{ A}_n(v)$ contains at least
${r_v(n)\over 2\varphi_v(n)}-1$ points i.e $|\mathcal{ A}_n(v)|\geq
{r_v(n)\over 2\varphi_v(n)}-1$.
\\\\
By conditional Jensen (recall $\Phi$ is convex), we get
$$E[\Phi\big(|L(Z)-\mu_n|\big)|V,U]\geq
\Phi\big(|E[L(Z)|V,U]-\mu_n|\big)=\Phi\big(|l(U,V)-\mu_n|\big).$$
Therefore (recall also that $\Phi$ is increasing)
\begin{align*}
E\Phi\big(|L(Z)-\mu_n|\big)&=E\Big(E[\Phi\big(L(Z)-\mu_n|\big)|V,U]\Big)\geq E\Phi\big(|l(U,V)-\mu_n|\big)\\
&\geq \sum_{v\in \V\cap \V^o}\sum_{u\in \U(v)}\Phi\big(|l(u,v)-\mu_n|\big)P(U=u|V=v)P(V=v)\\
& \geq \sum_{v\in \V\cap \V^o}\sum_{u\in \mathcal{ A}_n(v)}\Phi\big(|l(u,v)-\mu_n|\big)P(U=u|V=v)P(V=v)\\
&\geq  \sum_{v\in \V\cap \V^o}\sum_{u\in \mathcal{ A}_n(v)}\Phi\big(|l(u,v)-\mu_n|\big)\varphi_v(n)P(V=v)\\
&\geq \sum_{v\in \V\cap \V^o}\Phi\Big({\e_o
r_v(n)\over 8 \varphi_v(n)} \Big)|\mathcal{ A}_n(v)|\varphi_v(n)P(V=v)\\
&\geq \sum_{v\in \V\cap \V^o}\Phi\Big({\e_o
r_v(n)\over 8 \varphi_v(n)} \Big)\Big({r_v(n)\over 2\varphi_v(n)}-1\Big)\varphi_v(n)P(V=v)\\
&= \sum_{v\in \V\cap \V^o}\Phi\Big({\e_o r_v(n)\over 8
\varphi_v(n)}\Big)\Big({r_v(n)\over 2}-\varphi_v(n)\Big)P(V=v)
.\end{align*}
In particular, if $\varphi(n)=\sup_{v \in \V} \varphi(n)\to 0$ as $n \to \infty$ (that is,
$\varphi_v(n)$ converges to $0$ uniformly with respect to $v \in \V$) then
there exists $n_2$ such that $r_v(n)>{c\over
2}$ and $\varphi_v(n) \le c/8$ for all $n \ge n_2$, $v \in \V$. Thus, if $n \ge n_2$, then
$$E\Phi\big(|L(Z)-\mu_n|\big)\geq \Phi\Big({\e_o c\over 16
\varphi(n)}\Big)\Big({c\over 4}-\varphi(n)\Big)P(V\in  \V\cap \V^o)\geq
\Phi\Big({\e_o c\over 16 \varphi(n)}\Big){c\over 8}(P(V\in
\V)-\Delta^{1\over 2}_n).$$
If, in addition, $P(V\in \V)$ is bounded
away from zero, then for any constant $c_o$ satisfying $b_o c/8>c_o>0$ we can choose $n_3\ge n_2$ such that
for all $n \ge n_3$
\begin{align}\label{low}
E\Phi\big(|L(Z)-\mu_n|\big) \geq  \Phi\Big({\e_o c\over 16
\varphi(n)}\Big)c_o.\end{align}

\section*{Acknowledgements}

The authors acknowledge financial support from INDAM (Istituto
Nazionale di Alta Matematica), from  Estonian Science foundation
Grant no. 5822 and  from Estonian  institutional research funding
IUT34-5.

\bibliographystyle{alpha}


\end{document}